\documentclass[12pt,a4paper]{amsart}

\usepackage[all,2cell]{xy}

\newtheorem{theorem}[subsection]{Theorem}
\newtheorem*{theorem*}{Theorem}

\newtheorem{lemma}[subsection]{Lemma}
\newtheorem{proposition}[subsection]{Proposition}
\newtheorem{corollary}[subsection]{Corollary}

\newtheorem*{conjecture*}{Conjecture}

\newtheorem{remark}[subsection]{Remark}

\newcommand{\ie}{{\em i.e.}\ }
\newcommand{\confer}{{\em cf.}\ }

\newcommand{\ko}{\: , \;}

\newcommand{\ol}[1]{\overline{#1}}

\newcommand{\opname}[1]{\operatorname{\mathsf{#1}}}

\renewcommand{\mod}{\opname{mod}\nolimits}

\newcommand{\Mod}{\opname{Mod}\nolimits}
\newcommand{\Pcm}{\opname{Pcm}\nolimits}
\newcommand{\Lex}{\opname{Lex}\nolimits}

\newcommand{\per}{\opname{per}\nolimits}
\newcommand{\add}{\opname{add}\nolimits}
\newcommand{\op}{^{op}}

\newcommand{\hfd}[1]{\ch_{fd}(#1)/\ac_{fd}(#1)}

\renewcommand{\Im}{\opname{Im}\nolimits}
\newcommand{\Ker}{\opname{Ker}\nolimits}

\newcommand{\fpr}{\opname{fpr}\nolimits}

\newcommand{\bp}{\mathbf{p}}

\newcommand{\ac}{\mathcal{A}c}

\newcommand{\colim}{\opname{colim}}
\newcommand{\cok}{\opname{cok}\nolimits}
\newcommand{\im}{\opname{im}\nolimits}
\renewcommand{\ker}{\opname{ker}\nolimits}

\newcommand{\Z}{\mathbb{Z}}

\newcommand{\iso}{\stackrel{_\sim}{\rightarrow}}

\newcommand{\Hom}{\opname{Hom}}

\newcommand{\RHom}{\opname{RHom}}
\newcommand{\cHom}{\mathcal{H}\it{om}}

\newcommand{\Ext}{\opname{Ext}}

\newcommand{\End}{\opname{End}}

\newcommand{\ten}{\otimes}
\newcommand{\lten}{\overset{\boldmath{L}}{\ten}}

\newcommand{\ca}{{\mathcal A}}
\newcommand{\cb}{{\mathcal B}}
\newcommand{\cc}{{\mathcal C}}
\newcommand{\cd}{{\mathcal D}}

\newcommand{\cf}{{\mathcal F}}

\newcommand{\ch}{{\mathcal H}}

\newcommand{\cj}{{\mathcal J}}

\newcommand{\ct}{{\mathcal T}}
\newcommand{\cu}{{\mathcal U}}

\newcommand{\eps}{\varepsilon}

\renewcommand{\hat}[1]{\widehat{#1}}

\newcommand{\m}{\mathfrak{m}}
\newcommand{\del}{\partial}

\begin{document}

\title[Derived equivalence from mutations]{Derived equivalences from mutations 
of\\ quivers with potential}

\author{Bernhard Keller}
\address{Bernhard Keller\\
Universit\'e Paris Diderot -- Paris 7\\
UFR de Math\'ematiques\\
Institut de Math\'ematiques de Jussieu, UMR 7586 du CNRS \\
Case 7012\\
B\^{a}timent Chevaleret\\
75205 Paris Cedex 13\\
France } \email{keller@math.jussieu.fr}

\author{Dong Yang}
\address{
Dong Yang\\
Institut de Math\'ematiques de Jussieu, UMR 7586 du CNRS \\
Th\'eorie des groupes\\
Case 7012\\
2 place Jussieu\\
75251 Paris Cedex 05\\
France}
\email{yang@math.jussieu.fr}

\date{\today}
\maketitle

\bigskip

\begin{abstract}
  We show that Derksen-Weyman-Zelevinsky's mutations of quivers with
  potential yield equivalences of suitable $3$-Calabi-Yau triangulated
  categories. Our approach is related to that of Iyama-Reiten and
  `Koszul dual' to that of Kontsevich-Soibelman. It improves on
  previous work by Vit\'oria. In the appendix, the first-named author
  studies pseudo-compact derived categories of certain
  pseudo-compact dg algebras.
\end{abstract}

\tableofcontents

\section{Introduction}

\subsection{Previous work and context}
A quiver is an oriented graph. Quiver mutation is an elementary
operation on quivers. In mathematics, it first appeared in the
definition of cluster algebras invented by Fomin-Zelevinsky at the
beginning of this decade \cite{FominZelevinsky02} (though, with
hindsight, Gabrielov transformations \cite{Gabrielov73}
\cite{Ebeling81} may be seen as predecessors). In physics, it appeared
a few years before in Seiberg duality \cite{Seiberg95}. In this
article, we `categorify' quiver mutation by interpreting it in terms
of equivalences between derived categories. Our approach is based on
Ginzburg's Calabi-Yau algebras \cite{Ginzburg06} and on mutation of
quivers with potential. Let us point out that mutation of quivers with
potential goes back to the insight of mathematical physicists, \confer
for example section~5.2 of \cite{FengHananyHeIqbal03}, section~2.3 of
\cite{BeasleyPlesser01}, section~6 of
\cite{FengHananyHeUranga01} and section~3 of \cite{KlebanovStrassler00}.
Their methods were made rigorous by Derksen-Weyman-Zelevinsky
\cite{DerksenWeymanZelevinsky08}.  The idea of interpreting quiver
mutation via derived categories goes back at least to
Berenstein-Douglas \cite{BerensteinDouglas02} and was further
developed by Mukhopadhyay-Ray \cite{MukhopadhyayRay04} in physics and
by Bridgeland \cite{Bridgeland05}, Iyama-Reiten \cite{IyamaReiten06},
Vit\'oria \cite{Vitoria09}, Kontsevich-Soibelman
\cite{KontsevichSoibelman08} \ldots\ in mathematics. The main new
ingredient in our approach is Ginzburg's algebra
\cite{Ginzburg06} associated to a quiver with potential.  This
differential graded (=dg) algebra is concentrated in negative degrees.
Several previous approaches have only used its homology in degree~$0$,
namely the Jacobian algebra of \cite{DerksenWeymanZelevinsky08} or
vacualgebra of the physics literature.  By taking into account the full dg
algebra, we are able to get rid of the restrictive hypotheses that one
had to impose and to prove the homological interpretation in full
generality.

Our second source of inspiration and motivation is the theory which
links cluster algebras \cite{FominZelevinsky02} to representations
of quivers and finite-dimensional algebras, \confer~\cite{Keller08c}
for a recent survey. This theory is based on the discovery
\cite{MarshReinekeZelevinsky03} of the striking similarity between
the combinatorics of cluster algebras and those of tilting theory in
the representation theory of finite-dimensional algebras. The
cluster category, invented in
\cite{BuanMarshReinekeReitenTodorov06}, and the module category over
the preprojective algebra \cite{GeissLeclercSchroeer06} provided a
conceptual framework for this remarkable phenomenon. These
categories are triangulated respectively exact and $2$-Calabi-Yau.
Recall that a (triangulated or exact) category is $d$-Calabi-Yau if
there are non degenerate bifunctorial pairings
\[
\Ext^i(X,Y) \times \Ext^{d-i}(Y,X) \to k \ko
\]
where $k$ is the ground field. Iyama-Reiten \cite{IyamaReiten06}
used both $2$- and $3$-Calabi-Yau categories to categorify mutations
of different classes of quivers. We first learned about the
possibility of categorifying arbitrary quiver mutations using
$3$-Calabi-Yau categories from Kontsevich, who constructed the
corresponding categories using $A_\infty$-structures,
\confer~section~5 of \cite{Keller08d} for a sketch. This idea is
fully developed in section~8 of Kontsevich-Soibelman's work
\cite{KontsevichSoibelman08}. The Ginzburg dg algebras which we use
are Koszul dual to Kontsevich-Soibelman's $A_\infty$-algebras. Since
a Ginzburg dg algebra is concentrated in negative degrees, its
derived category has a canonical $t$-structure given by the standard
truncation functors. The heart is equivalent to the module category
over the zeroth homology of the Ginzburg algebra, \ie~over the
Jacobian algebra. This allows us
to establish the link to previous approaches, especially to the one
based on the cluster category and the one based on decorated
representations \cite{MarshReinekeZelevinsky03}
\cite{DerksenWeymanZelevinsky08} \cite{DerksenWeymanZelevinsky09}.

\subsection{The classical paradigm and the main result}
Like Derksen--Weyman--Zele\-vinsky \cite{DerksenWeymanZelevinsky08},
we model our approach on the reflection functors introduced by
Bernstein--Gelfand--Ponomarev \cite{BernsteinGelfandPonomarev73}.
Let us recall their result in a modern form similar to the one we
will obtain: Let $Q$ be a finite quiver and $i$ a {\em source} of
$Q$, \ie a vertex without incoming arrows. Let $Q'$ be the quiver
obtained from $Q$ by reversing all the arrows going out from $i$.
Let $k$ be a field, $kQ$ the path algebra of $Q$ and $\cd(kQ)$ the
derived category of the category of all right $kQ$-modules. For a
vertex $j$ of $Q'$ respectively $Q$, let $P'_j$ respectively $P_j$
be the projective indecomposable associated with the vertex $j$
(\ie the right ideal generated by the idempotent $e_j$ of the
respective path algebra). Then the main result of
\cite{BernsteinGelfandPonomarev73} reformulated in terms of derived
categories following Happel \cite{Happel87} says that there is a
canonical triangle equivalence
\[
F: \cd(kQ') \to \cd(kQ)
\]
which takes $P'_j$ to $P_j$ for $j\neq i$ and
$P'_i$ to the cone over the morphism
\[
P_i \to \bigoplus P_j
\]
whose components are the left multiplications by all arrows going
out from $i$. In Theorem~\ref{T:mainthm}, we will obtain an
analogous result for the mutation of a quiver with potential $(Q,W)$
at an {\em arbitrary} vertex $i$, where the r\^ole of the quiver
with reversed arrows is played by the quiver with potential $(Q',
W')$ obtained from $(Q,W)$ by mutation at $i$ in the sense of
Derksen-Weyman-Zelevinsky \cite{DerksenWeymanZelevinsky08}. The
r\^ole of the derived category $\cd(kQ)$ will be played by the
derived category $\cd(\Gamma)$ of the complete differential graded
algebra $\Gamma=\Gamma(Q,W)$ associated with $(Q,W)$ by Ginzburg
\cite{Ginzburg06}. Let us point out that our result is analogous to
but not a generalization of Bernstein-Gelfand-Ponomarev's since even
if the potential $W$ vanishes, the derived category $\cd(\Gamma)$ is
{\em not} equivalent to $\cd(kQ)$. Notice that we work
with completed path algebras because this is the setup
in which Derksen-Weyman-Zelevinsky's results are proved.
Nevertheless, with suitable precautions, one can prove similar 
derived equivalences for non complete Ginzburg algebras, 
\confer \cite{Keller08a}.

\subsection{Plan of the paper and further results}
In section~\ref{S:preliminaries}, we first recall
Derksen-Weyman-Zelevinsky's results on mutations of quivers with
potential. In section~\ref{ss:ginzburgalg}, we recall the definition
of the Ginzburg dg algebra associated to a quiver with potential and
show that reductions of quivers with potential yield
quasi-isomorphisms between Ginzburg dg algebras. Then we recall the
basics on derived categories of differential graded algebras in
sections \ref{ss:derived-categories} and
\ref{ss:cofibrant-resolutions}. In
section~\ref{ss:finite-dimensional-dg-modules}, we prove the useful
Theorem~\ref{thm:findim-compact}, which could be viewed as an
`Artin-Rees lemma' for derived categories of suitable complete dg
algebras.  In section~\ref{S:the-equivalence}, we state and prove the
main result described above. In
section~\ref{s:nearly-Morita-equivalence}, we establish a link with
cluster categories in the sense of Amiot~\cite{Amiot08a}, who
considered the Jacobi-finite case, and Plamondon \cite{Plamondon09},
who is considering the general case. We use this to give another proof
of the fact that neighbouring Jacobian algebras are nearly Morita
equivalent in the sense of Ringel \cite{Ringel07}.  This was
previously proved in \cite{BuanIyamaReitenSmith08} and
\cite{DerksenWeymanZelevinsky09}.  We include another proof to
illustrate the links between our approach and those of these two
papers. In section~\ref{s:tilting-Jacobian-algebras}, we compare the
canonical $t$-structures in the derived categories associated with a
quiver with potential and its mutation at a vertex. In general, the
hearts of these $t$-structures are not related by a tilt in the sense
of Happel-Reiten-Smal\o\ \cite{HappelReitenSmaloe96}. However, this is
the case if the homology of the Ginzburg dg algebra is concentrated in
degree~$0$. In section~\ref{s:stability-under-mutation}, we show that
this last condition is preserved under mutation. In a certain sense,
this explains the beautiful results of Iyama-Reiten
\cite{IyamaReiten06} on mutation of $3$-Calabi-Yau algebras
(concentrated in degree~$0$): In fact, these algebras are
quasi-isomorphic to the Ginzburg algebras associated to the
corresponding quivers with potential. Strictly speaking, the classical
derived category is not well-adapted to topological dg algebras like
the completed Ginzburg dg algebra which we use. In the appendix, the
first-named author recalls the basics of the theory of pseudocompact
algebras from \cite{Gabriel62}. As shown in
Corollary~\ref{cor:fin-dim-modules-determine-Jacobi-algebra}, the
theory immediately yields a positive answer to question~12.1 of
\cite{DerksenWeymanZelevinsky08}, where the authors asked whether a
Jacobian algebra is determined by its category of finite-dimensional
modules. In section~\ref{ss:extension-to-dg-setting}, the theory of
pseudocompact algebras and modules is then adapted, under certain
boundedness and regularity assumptions, to pseudocompact differential
graded algebras. The analogue of the main theorem also holds in this
setting (section~\ref{ss:pseudocompact-main-theorem}).  Moreover, the
results of Amiot \cite{Amiot08a} on Jacobi-finite quivers with
potentials can be extended to potentials in complete path algebras as
we see in section~\ref{ss:Jacobi-finite-quivers-with-potentials}.

\subsection{Acknowledgments} The first-named author warmly thanks
D.~Simson for his letter \cite{Simson07a} where he points out the
relevance of coalgebras and comodules as well as pseudocompact
algebras and modules in the study of the representations of quivers
with potential. He is deeply grateful to M.~Kontsevich for sharing his
insights and to M.~Kontsevich and Y.~Soibelman for a preliminary
version of \cite{KontsevichSoibelman08}. He thanks T.~Bridgeland
for inspiring talks and conversations and A.~King for kindly
pointing out references to the physics literature. The second-named author
gratefully acknowledges a postdoctoral fellowship granted by the
Universit\'e Pierre et Marie Curie, during which this work was carried
out. He is deeply grateful to the first-named author for patiently
leading him to the beautiful world of derived categories and derived
equivalences and for his consistent support.

\section{Preliminaries}\label{S:preliminaries}

\subsection{Quivers with potential} \label{ss:quivers-with-potential}
We follow \cite{DerksenWeymanZelevinsky08}. 
Let $k$ be a field. Let
$Q$ be a finite quiver (possibly with loops and 2-cycles). We denote
its set of vertices by $Q_0$ and its set of arrows by $Q_1$. For an
arrow $a$ of $Q$, let $s(a)$ denote its source and $t(a)$ denote its
target. The lazy path corresponding to a vertex $i$ will be denoted
by $e_i$.

The {\em complete path algebra $\widehat{kQ}$} is the completion of
the path algebra $kQ$ with respect to the ideal generated by the
arrows of $Q$. Let $\mathfrak{m}$ be the ideal of $\widehat{kQ}$
generated by the arrows of $Q$. A {\em potential} on $Q$ is
an element of the closure $Pot(kQ)$ of the space generated by all non trivial cyclic
paths of $Q$. We say two potentials are \emph{cyclically
equivalent} if their difference is in the closure of the space
generated by all differences $a_1 \ldots a_s - a_2 \ldots a_s a_1$,
where $a_1 \ldots a_s$ is a cycle.

For a path $p$ of $Q$, we define $ \del_p : Pot(kQ) \to \widehat{kQ} $
as the unique continuous linear map which takes a cycle $c$ to the sum
$ \sum_{c=u p v} vu $ taken over all decompositions of the cycle $c$
(where $u$ and $v$ are possibly lazy paths). Obviously two cyclically
equivalent potentials have the same image under $\del_p$.  If $p=a$ is
an arrow of $Q$, we call $\del_a$ the {\em cyclic derivative with
  respect to $a$}.

\begin{remark} The space of cyclic equivalence classes of elements in $Pot(kQ)$
identifies with the {\em reduced continuous Hochschild homology}
$HH_0^{red}(\widehat{kQ})$, the quotient of $\widehat{kQ}$ by the closure
of the subspace generated by the lazy paths and the commutators of elements of $kQ$.
\end{remark}

Let $W$ be a potential on $Q$ such that $W$ is in $\mathfrak{m}^2$ and
no two cyclically equivalent cyclic paths appear in the decomposition
of $W$ (\cite[(4.2)]{DerksenWeymanZelevinsky08}). Then the pair
$(Q,W)$ is called a \emph{quiver with potential}. Two quivers with
potential $(Q,W)$ and $(Q',W')$ are \emph{right-equivalent} if $Q$ and
$Q'$ have the same set of vertices and there exists an algebra
isomorphism $\varphi:\widehat{kQ}\rightarrow \widehat{kQ'}$ whose
restriction on vertices is the identity map and $\varphi(W)$ and
$W'$ are cyclically equivalent. Such an isomorphism $\varphi$ is
called a \emph{right-equivalence}.

\emph{The Jacobian algebra} of a quiver with potential $(Q,W)$,
denoted by $J(Q,W)$, is the quotient of the complete path algebra
$\widehat{kQ}$ by the closure of the ideal generated by $\del_a W$,
where $a$ runs over all arrows of $Q$. It is clear that two
right-equivalent quivers with potential have isomorphic Jacobian
algebras. A quiver with potential is called \emph{trivial} if its
potential is a linear combination of cycles of length $2$ and its
Jacobian algebra is the product of copies of the base field $k$.  A
quiver with potential is called \emph{reduced} if $\del_a W$ is
contained in $\mathfrak{m}^2$ for all arrows $a$ of $Q$.

Let $(Q',W')$ and $(Q'',W'')$ be two quivers with potential such that
$Q'$ and $Q''$ have the same set of vertices. Their \emph{direct sum},
denoted by $(Q',W')\oplus (Q'',W'')$, is the new quiver with potential
$(Q,W)$, where $Q$ is the quiver whose vertex set is the same as the
vertex set of $Q'$ (and $Q''$) and whose arrow set the union of the
arrow set of $Q'$ and the arrow set of $Q''$, and $W=W'+W''$.

\begin{theorem}\label{T:DWZreduction} In the following statements,
all quivers are quivers without loops.
\begin{itemize}
\item[a)] \cite[Proposition 4.4]{DerksenWeymanZelevinsky08}
A quiver with potential
is trivial if and only if it is right-equivalent, via a right-equivalence which takes arrows to linear combinations of arrows, to a quiver with potential $(Q,W)$ such that $Q_1$ consists
of $2n$ distinct arrows
\[Q_1=\{a_1,b_1,\ldots,a_n,b_n\},\]
where each $a_ib_i$ is a 2-cycle, and
\[W=\sum_{i=1}^n a_ib_i.\]

\item[b)] \cite[Proposition 4.5]{DerksenWeymanZelevinsky08} Let $(Q',W')$ be a quiver with potential and
$(Q'',W'')$ be a trivial quiver with potential. Let $(Q,W)$ be their direct sum. Then
the canonical embedding $\widehat{kQ'}\rightarrow \widehat{kQ}$
induces an isomorphism of Jacobian algebras $J(Q',W')\rightarrow
J(Q,W)$.

\item[c)] \cite[Theorem 4.6]{DerksenWeymanZelevinsky08} For any quiver with potential $(Q,W)$, there exist a
trivial quiver with potential $(Q_{\mathrm{tri}},W_{\mathrm{tri}})$ and a reduced
quiver with potential $(Q_{\mathrm{red}},W_{\mathrm{red}})$ such that $(Q,W)$ is
right-equivalent to the direct sum
$(Q_{\mathrm{tri}},W_{\mathrm{tri}})\oplus
(Q_{\mathrm{red}},W_{\mathrm{red}})$. Furthermore, the
right-equivalence class of each of
$(Q_{\mathrm{tri}},W_{\mathrm{tri}})$ and
$(Q_{\mathrm{red}},W_{\mathrm{red}})$ is determined by the
right-equivalence class of $(Q,W)$. We call
$(Q_{\mathrm{tri}},W_{\mathrm{tri}})$ and
$(Q_{\mathrm{red}},W_{\mathrm{red}})$ respectively \emph{the trivial
part} and \emph{the reduced part} of $(Q,W)$.
\end{itemize}
\end{theorem}

\subsection{Mutations of quivers with potential}\label{S:mutation}
We follow \cite{DerksenWeymanZelevinsky08}.
Let $(Q,W)$ be a quiver with potential. Let $i$ be a vertex of $Q$.
Assume the following conditions:

\begin{itemize}

\item[(c1)] the quiver $Q$ has no loops;

\item[(c2)] the quiver $Q$ does not have 2-cycles at $i$;

\item[(c3)] no cyclic path occurring in the expansion of $W$ starts
and ends at $i$ (\cite[(5.2)]{DerksenWeymanZelevinsky08}).
\end{itemize}
Note that under the condition (c1), any potential is cyclically
equivalent to a potential satisfying (c3).
We define a new quiver with potential $\tilde{\mu}_{i}(Q,W)=(Q',W')$ as follows. The
new quiver $Q'$ is obtained from $Q$ by
\begin{itemize}
\item[Step 1] For each arrow $\beta$ with target $i$ and each
arrow $\alpha$ with source $i$, add a new arrow $[\alpha\beta]$ from the source of
$\beta$ to the target of $\alpha$.
\item[Step 2] Replace each arrow $\alpha$ with source or target $i$ with
an arrow $\alpha^\star$ in the opposite direction.
\end{itemize}
The new potential $W'$ is the sum of two potentials $W'_1$ and
$W'_2$. The potential $W'_1$ is obtained from $W$ by replacing each composition
$\alpha\beta$ by $[\alpha\beta]$, where $\beta$ is an arrow
with target $i$. The potential $W'_2$ is given by
\[W'_2=\sum_{\alpha,\beta}[\alpha\beta]\beta^{\star}\alpha^{\star},\]
where the sum ranges over all pairs of arrows $\alpha$ and $\beta$
such that $\beta$ ends at $i$ and $\alpha$ starts at $i$. It is easy
to see that $\tilde{\mu}_i(Q,W)$ satisfies (c1) (c2) (c3).
We define $\mu_i(Q,W)$ as the reduced part of $\tilde{\mu}_i(Q,W)$,
and call $\mu_i$ the \emph{mutation at the vertex~$i$}.

\begin{theorem}
\begin{itemize}
 \item [a)] \cite[Theorem 5.2]{DerksenWeymanZelevinsky08} The
right-equivalence class of $\tilde{\mu}_i(Q,W)$ is determined by the
right-equivalence class of $(Q,W)$.

\item[b)] \cite{DerksenWeymanZelevinsky08} The quiver with potential $\tilde\mu_i^2(Q,W)$ is right-equivalent to the direct sum of $(Q,W)$ with a trivial quiver with potential.

\item[c)] \cite[Theorem 5.7]{DerksenWeymanZelevinsky08} The correspondence
$\mu_i$ acts as an involution on the right-equivalence classes of
reduced quivers with potential satisfying (c1) and (c2).
\end{itemize}
\end{theorem}

\subsection{The complete Ginzburg dg algebra}\label{ss:ginzburgalg}
Let $(Q,W)$ be a quiver with potential.
The {\em complete Ginzburg dg algebra $\hat{\Gamma}(Q,W)$} is
constructed as follows \cite{Ginzburg06}: Let $\tilde{Q}$ be the
graded quiver with the same vertices as $Q$ and whose arrows are
\begin{itemize}
\item the arrows of $Q$ (they all have degree~$0$),
\item an arrow $a^*: j \to i$ of degree $-1$ for each arrow $a:i\to j$ of $Q$,
\item a loop $t_i : i \to i$ of degree $-2$ for each vertex $i$
of $Q$.
\end{itemize}
The underlying graded algebra of $\hat{\Gamma}(Q,W)$ is the
completion of the graded path algebra $k\tilde{Q}$ in the category
of graded vector spaces with respect to the ideal generated by the
arrows of $\tilde{Q}$. Thus, the $n$-th component of
$\hat{\Gamma}(Q,W)$ consists of elements of the form
$\sum_{p}\lambda_p p$, where $p$ runs over all paths of degree $n$.
The differential of $\hat{\Gamma}(Q,W)$ is the unique continuous
linear endomorphism homogeneous of degree~$1$ which satisfies the
Leibniz rule
\[
d(uv)= (du) v + (-1)^p u dv \ko
\]
for all homogeneous $u$ of degree $p$ and all $v$, and takes the
following values on the arrows of $\tilde{Q}$:
\begin{itemize}
\item $da=0$ for each arrow $a$ of $Q$,
\item $d(a^*) = \del_a W$ for each arrow $a$ of $Q$,
\item $d(t_i) = e_i (\sum_{a} [a,a^*]) e_i$ for each vertex $i$ of $Q$, where
$e_i$ is the lazy path at $i$ and the sum runs over the set of
arrows of $Q$.
\end{itemize}

\begin{remark} The authors thank N.~Broomhead for pointing out
that Ginzburg's original definition in section~5.3
of \cite{Ginzburg06} concerns the case where the potential $W$ lies
in the non completed path algebra $kQ$ and that he completes
with respect to the ideal generated by the cyclic derivatives
of $W$ rather than the ideal generated by the arrows of $Q$.
\end{remark}

The following lemma is an easy consequence of the definition.

\begin{lemma}
  Let $(Q,W)$ be a quiver with potential. Then the Jacobian algebra
  $J(Q,W)$ is the 0-th cohomology of the complete Ginzburg dg algebra
  $\hat{\Gamma}(Q,W)$, \ie
\[
J(Q,W)=H^0\hat{\Gamma}(Q,W).
\]
\end{lemma}
\bigskip

Let $(Q,W)$ and $(Q',W')$ be two quivers with potential. Let
$\Gamma=\hat{\Gamma}(Q,W)$ and $\Gamma'=\hat\Gamma(Q',W')$ be the
associated complete Ginzburg dg algebras.

\begin{lemma} Assume $(Q,W)$ and $(Q',W')$ are right-equivalent. Then the
complete Ginzburg dg algebras $\Gamma$ and $\Gamma'$ are isomorphic
to each other.
\end{lemma}

\begin{proof}
Let $\varphi$ be a right-equivalence from $(Q,W)$ to $(Q',W')$. Let
$\varphi_{*}:\Gamma\to\Gamma'$ be the unique continuous algebra
homomorphism such that its restriction to $Q_0$ is the identity and
for $\rho\in Q_1$ and $i\in Q_0$
\begin{eqnarray*}
\varphi_{*}(\rho)&=&\varphi(\rho)\\
\varphi_{*}(\rho^*)&=&\sum_{\rho'\in Q'_1}\sum_{p}\sum_{1\leq j\leq
l(p)}b_{p,\rho'}\delta_{\rho,\rho_j^p}\varphi(\rho_{j+1}^p\cdots\rho_{l(p)}^p)\rho'^*\varphi(\rho_1^p\cdots\rho_{j-1}^p)\\
\varphi_{*}(t_i)&=&t'_i,
\end{eqnarray*}
where the middle summation is over all paths
$p=\rho_1^p\cdots\rho_{l(p)}^p$ ($\rho_j^p\in Q_1$) of $Q$, $\delta$
is the Kronecker symbol, and for $\rho'\in Q'_1$
\[\varphi^{-1}(\rho')=\sum_{p:\text{path of }Q}b_{p,\rho'}p,\]
and $t'_i$ is the loop of $\tilde{Q'}$ of degree -2 at the vertex $i$.

It is straightforward to prove that $\varphi_*$ commutes with the
differentials. Thus $\varphi_*$ is a homomorphism of dg algebras.
Moreover, $\varphi_*$ induces a linear map from $k\tilde{Q}_1$ (the
$k$-vector space with basis the arrows of $\tilde{Q}$) to
$k\tilde{Q'}_1$, whose matrix is invertible. Indeed, this matrix is
quasi-diagonal with diagonal blocks the matrix of the linear map
from $kQ_1$ to $kQ'_1$ induced from $\varphi$, the matrix
$(b_{\rho,\rho'})_{\rho\in Q_1,\rho'\in Q'_1}$ (which we read from
the definition of $\varphi_*(\rho^*)$), and the identity matrix of
size the cardinality of $Q_0$ (which we read from the definition of
$\varphi(t_i)$), all the three of which are invertible. Therefore
$\varphi_*$ defined as above is an isomorphism of dg algebras from
$\Gamma$ to $\Gamma'$.
\end{proof}

\begin{lemma}\label{L:reduction} Suppose that $(Q,W)$ is the direct sum of $(Q',W')$ and $(Q'',W'')$
and that $(Q'',W'')$ is trivial. Then the canonical projection
$\Gamma\to\Gamma'$ induces quasi-isomorphisms
\[
\Gamma/\m^n \to \Gamma'/\m'^n
\]
for all $n\geq 1$ and a quasi-isomorphism $\Gamma\to \Gamma'$.
\end{lemma}

\begin{proof} Let $K$ be the kernel of the projection from $\Gamma$
onto $\Gamma'$. For $i\geq 1$, let $K_i\subset K$ be $\m^i\cap K$.
Thus, $K_i$ is the topological span of the paths of length $\geq i$
with at least one arrow in $Q''$ or $Q''^*$. Then the $K_i$ form a
decreasing filtration of $K$ and the complex $K_i/K_{i+1}$ is
isomorphic to
\[
(U\oplus V)^{\ten_R i}/U^{\ten_R i}\ko
\]
where $R=\prod_{i\in Q_0} k$, $V=kQ_1''\oplus kQ_1''^\ast$ as a
complex endowed with the differential given by $W''$ (thus, $V$ is
contractible) and $U=kQ_1'\oplus kQ_1'^\ast\oplus k\{t_i|i\in Q_0\}$
as a complex with the differential which is zero on $k\{t_i|i\in
Q_0\}$ and, on $kQ_1'\oplus kQ_1'^\ast$, is given by the projection
of $W$ onto the subspace of cycles of length $2$. Thus $K_i/K_{i+1}$
is contractible. It follows that $K/K_n$ is contractible for each
$n\geq 1$, which implies the first assertion. The second one now
follows from the Mittag-Leffler lemma~\ref{lemma:Mittag-Leffler}.
\end{proof}

\begin{remark} As a consequence of Lemma~\ref{L:reduction}, the
canonical projection $\hat{kQ}\rightarrow \hat{kQ'}$ induces an
isomorphism of Jacobian algebras $J(Q,W)\rightarrow J(Q',W')$, which
is the inverse of the isomorphism in Theorem~\ref{T:DWZreduction}
b).\end{remark}

\subsection{Derived categories} \label{ss:derived-categories}
We follow~\cite{Keller94}.
Let $A$ be a dg $k$-algebra with differential
$d_A$ (all differentials are of degree~1).
A \emph{(right) dg module} $M$ over $A$ is a graded $A$-module
equipped with a differential $d_M$ such that
\[d_M(ma)=d_M(m)a+(-1)^{|m|}md_A(a)\]
where $m$ in $M$ is homogeneous of degree $|m|$, and $a\in A$.

Given two dg $A$-modules $M$ and $N$, we define the \emph{morphism
complex} to be the graded $k$-vector space $\mathcal{H}om_A(M,N)$
whose $i$-th component $\mathcal{H}om_A^i(M,N)$ is the subspace of
the product $\prod_{j\in\mathbb{Z}}\Hom_k(M^j,N^{j+i})$ consisting
of morphisms $f$ such that \[f(ma)=f(m)a,\] for all $m$ in $M$ and
all $a$ in $A$, together with the differential $d$ given by
\[d(f)=f\circ d_M-(-1)^{|f|}d_N\circ f\]
for a homogeneous morphism $f$ of degree $|f|$.

The category $\mathcal{C}(A)$ of dg $A$-modules is the category
whose objects are the dg $A$-modules, and whose morphisms are the
0-cycles of the morphism complexes. This is an abelian category and
a Frobenius category for the conflations which are split exact as
sequences of graded $A$-modules. Its stable category
$\mathcal{H}(A)$ is called the \emph{homotopy category} of dg
$A$-modules, which is equivalently defined as the category whose
objects are the dg $A$-modules and whose morphism spaces are the
$0$-th homology groups of the morphism complexes. The homotopy
category $\mathcal{H}(A)$ is a triangulated category whose suspension
functor $\Sigma$ is the shift of dg modules $M \mapsto M[1]$. The \emph{derived
category} $\mathcal{D}(A)$ of dg $A$-modules is the localization of
$\mathcal{H}(A)$ at the full subcategory of acyclic dg $A$-modules.
A short exact sequence
\[\xymatrix{0\ar[r]&M\ar[r]&N\ar[r]&L\ar[r]&0}\]
in $\mathcal{C}(A)$ yields a triangle
\[\xymatrix{M\ar[r]&N\ar[r]&L\ar[r]&\Sigma M}\]
in $\mathcal{D}(A)$.
A dg $A$-module $P$ is \emph{cofibrant} if
\[\Hom_{\mathcal{C}(A)}(P,L)\stackrel{s_*}{\longrightarrow}\Hom_{\mathcal{C}(A)}(P,M)\]
is surjective for each quasi-isomorphism $s:L\rightarrow M$ which is
surjective in each component. We use the term ``cofibrant'' since
these are actually the objects which are cofibrant for a certain
structure of Quillen model category on the category $\cc(A)$,
\confer \cite[Theorem 3.2]{Keller06d}.

\begin{proposition} A dg $A$-module is cofibrant if and only
if it is a direct summand of a dg module $P$ which admits a filtration
\[
\ldots\subset F_{p-1}\subset F_{p}\subset\ldots\subset P,~~p\in\mathbb{Z}
\]
in $\mathcal{C}(A)$ such that
\begin{itemize}
 \item[(F1)] we have $F_p=0$ for all $p\ll 0$, and $P$ is the union of the $F_p$, $p\in\mathbb{Z}$;
 \item[(F2)] the inclusion morphism $F_{p-1}\subset F_{p}$ splits in the category of graded $A$-modules, for all $p\in\mathbb{Z}$;
\item[(F3)] the subquotient $F_p/F_{p-1}$ is isomorphic in $\mathcal{C}(A)$ to a direct summand of a direct sum of shifted copies of $A$, for all $p\in\mathbb{Z}$.
\end{itemize}
\end{proposition}

Let $P$ be a cofibrant dg $A$-module. Then the canonical map
\[\Hom_{\mathcal{H}(A)}(P,N)\rightarrow\Hom_{\mathcal{D}(A)}(P,N)\]
is bijective for all dg $A$-modules $N$. The canonical projection from
$\mathcal{H}(A)$ to $\mathcal{D}(A)$ admits a left adjoint functor
$\mathbf{p}$ which sends a dg $A$-module $M$ to a cofibrant dg
$A$-module $\mathbf{p}M$ quasi-isomorphic to $M$. We call $\mathbf{p}M$ the
\emph{cofibrant resolution} of $M$. In other words, we have
\[
\Hom_{\mathcal{D}(A)}(M,N)=\Hom_{\mathcal{H}(A)}(\mathbf{p}M,N)=H^0(\mathcal{H}om_A(\mathbf{p}M,N)).
\]
The \emph{perfect derived category} $\per(A)$ is the smallest full
subcategory of $\mathcal{D}(A)$ containing $A$ which is stable under
taking shifts, extensions and direct summands. An object $M$ of
$\mathcal{D}(A)$ belongs to $\per(A)$ if and only if it is {\em compact},
\ie the functor $\Hom_{\mathcal{D}(A)}(M,?)$ commutes with arbitrary
(set-indexed) direct sums, \confer section~5 of \cite{Keller94} 
and \cite{Neeman99}. The \emph{finite dimensional derived
  category} $\mathcal{D}_{fd}(A)$ is the full subcategory of
$\mathcal{D}(A)$ consisting of those dg $A$-modules whose homology is
of finite total dimension, or equivalently those dg $A$-modules $M$
such that the sum of the dimensions of the spaces
\[
\Hom_{\mathcal{D}(A)}(P,\Sigma^p M), \ko p\in \Z\ko
\]
is finite for any $P$ in $\per(A)$.

\subsection{Cofibrant resolutions of simples over a tensor algebra}
\label{ss:cofibrant-resolutions}
Let $Q$ be a finite graded quiver, and $\hat{kQ}$ the complete path
algebra. Thus, the $i$th component of the graded algebra $\hat{kQ}$
is the completion of the space of paths of total degree $i$ with
respect to the descending filtration by path length.
Let $\mathfrak{m}$ be the two-sided ideal of $\hat{kQ}$
generated by arrows of $Q$.  Let $A=(\hat{kQ},d)$ be a topological
dg algebra whose differential takes each arrow of $Q$ to an element
of $\mathfrak{m}$.

For a vertex $i$ of $Q$, let $P_i=e_iA$, and let $S_i$ be the simple
module corresponding to $i$. Then in $\mathcal{C}(A)$ we have a short
exact sequence
\[
\xymatrix{0\ar[r] & \ker(\pi)\ar[r] & P_i\ar[r]^{\pi} & S_i\ar[r] & 0,}
\]
where $\pi$ is the canonical projection from $P_i$ to $S_i$.
Explicitly, we have
\[
\ker(\pi)=\sum_{\rho\in Q_1:t(\rho)=i}\rho P_{s(\rho)}=\bigoplus_{\rho\in Q_1:t(\rho)=i}\rho P_{s(\rho)}
\]
with the induced differential. Notice that the direct sum
decomposition only holds in the category of graded $A$-modules and
that, if $d(\rho)$ does not vanish, the summand $\rho P_{s(\rho)}$ is
not stable under the differential of $\ker(\pi)$. The simple module
$S_i$ is quasi-isomorphic to
\[P=\mathrm{Cone}(\ker(\pi)\rightarrow P_i),\]
whose underlying graded space is
\[\bigoplus_{\rho\in Q_1:t(\rho)=i}\Sigma\rho P_{s(\rho)}\oplus P_i.\]
Let $n$ be the maximum of the degrees of the arrows ending in $i$. Let
\[F_{p}=0 \text{ for } p<-n,\qquad F_{-n}=P_i,\]
and for $p>-n$, let
\[F_p=\bigoplus_{\rho}\Sigma\rho P_{s(\rho)}\oplus P_i,
\]
where the sum is taken over the arrows $\rho$ of $Q$ ending in $i$
and of degree at least $1-p$. Then each $F_p$, $p\in\mathbb{Z}$, is
a dg $A$-submodule of $P$. It is easy to check that (F1) (F2) (F3)
hold. Therefore, $P$ is a cofibrant dg $A$-module, and hence it is a
cofibrant resolution of $S_i$. In particular, $S_i$ belongs to the
perfect derived category $\per(A)$.

\begin{lemma} Let $i$ and $j$ be two vertices of $Q$, and let $n$ be an integer.
Let $a_{ij}^n$ denote the number of arrows of $Q$ starting from $j$
and ending at $i$ and of degree $-n+1$. Then the dimension over $k$
of $\Hom_{\cd(A)}(S_i,\Sigma^n S_j)$ equals $a_{ij}^n+1$ if $i=j$ and
$n=0$, and equals $a_{ij}^n$ otherwise.
\end{lemma}
\begin{proof} First one notices that the morphism complex
$\cHom_A(P_i,S_j)$ is 1-dimensional and concentrated in degree~$0$
if $i=j$ and vanishes otherwise. Now let $P$ be the cofibrant
resolution of $S_i$ constructed as above. 
By the assumption on
the differential of $A$, the differential of $\cHom_A(P,S_j)$
vanishes. Therefore, we have
\[
\Hom_{\cd(A)}(S_i,\Sigma^n S_j)=\Hom_{\ch(A)}(P,\Sigma^n S_j)=H^n\cHom_A(P,S_j)=\cHom_A^n(P,S_j),
\]
which has a basis
\[\{\pi_{\rho}|\rho\in
Q_1, s(\rho)=j,t(\rho)=i,|\rho|=-n+1\}\cup\{\pi_i\}\] if $i=j$ and
$n=0$, and
 \[\{\pi_{\rho}|\rho\in
Q_1, s(\rho)=j,t(\rho)=i,|\rho|=-n+1\}\] otherwise,  where $\pi_i$
is the projection from $P_i$ to $S_i$, and for an arrow $\rho$ of
$Q$ starting from $j$, $\pi_{\rho}$ is the projection from
$\Sigma\rho P_{j}\cong\Sigma^n P_j$ to $S_j$. The assertion follows immediately.
\end{proof}

\subsection{The perfect derived category is Krull-Schmidt}
\label{ss:perfect-Krull-Schmidt} As in
section~\ref{ss:cofibrant-resolutions}, let $Q$ be a finite graded
quiver, $\hat{kQ}$ the complete path algebra, $\mathfrak{m}$ the
two-sided ideal of $\hat{kQ}$ generated by the arrows of $Q$ and
$A=(\hat{kQ}, d)$ a topological dg algebra whose differential takes
each arrow of $Q$ to an element of $\mathfrak{m}$. 

\begin{lemma} The perfect derived category $\per(A)$ is
  Krull--Schmidt, \ie any object in $\per(\Gamma)$ is isomorphic to a
  finite direct sum of objects whose endomorphism rings are local.
\end{lemma}

\begin{proof} Clearly the category of finitely generated projective
  graded $A$-modules is Krull-Schmidt and its radical is the ideal
  generated by the morphisms $A[i] \to A[i]$, $i\in \Z$, given by the
  left multiplication by an element of the ideal $\m$. Moreover, in
  this category, each endomorphism $f$ of an object $P$ admits a
  `Fitting-Jordan decomposition', which can be constructed as follows:
  Let $i>0$. Then the morphism $f_i$ induced by $f$ in $P/P\m^i$
  admits a Fitting-Jordan decomposition into the direct sum of
  $I_i=\Im f_i^{m_i}$ and $K_i=\Ker f_i^{m_i}$ for some integer $m_i$.
  Clearly, by induction, we can even construct an increasing sequence
  of integers $m_i$, $i>0$, with this property.  Then we have natural
  projection morphisms $K_{i+1} \to K_i$ and $I_{i+1} \to I_i$ and
  clearly $P$ decomposes into the direct sum of the inverse limits
  $K_\infty$ and $I_\infty$ and $f$ acts nilpotently in each
  finite-dimensional quotient of $K_\infty$ and invertibly in
  $I_\infty$. Thus, the morphism induced by $f$ in $K_\infty$ belongs
  to the radical and the morphism induced in $I_\infty$ is invertible.
  Now we consider the category of strictly perfect dg $A$-modules, \ie
  the full subcategory of the category of dg $A$-modules obtained as
  the closure of $A$ under shifts, direct factors and extensions which
  split in the category of graded $A$-modules (equivalently: cones).
  If $P$ belongs to this category, its underlying graded $A$-module is
  finitely generated projective.  It follows that $P$ decomposes into
  a finite sum of indecomposable dg $A$-modules. Now assume that $P$
  is indecomposable as a dg $A$-module and that $f$ an endomorphism of
  $P$ in the category of dg $A$-modules. If we apply the
  Fitting-Jordan decomposition to the morphism of graded modules
  underlying $f$, we find a decomposition of $P$ as the direct sum of
  $K_\infty$ and $I_\infty$ and the construction shows that these are
  dg submodules of $P$, that $f$ acts by a radical morphism in
  $K_\infty$ and by an invertible morphism in $I_\infty$. Since $P$ is
  indecomposable, $f$ is radical or invertible. We conclude that the
  category of strictly perfect dg $A$-modules is a Krull-Schmidt
  category. It follows that the canonical functor from the category of strictly
  perfect dg $A$-modules to the perfect derived category is essentially 
  surjective and that the perfect derived category is a
  Krull-Schmidt category since it is the quotient of the category of
  strictly perfect dg $A$-modules by the ideal of morphisms factoring
  through contractible strictly perfect dg $A$-modules.
\end{proof}

\subsection{On finite-dimensional dg modules}
\label{ss:finite-dimensional-dg-modules} As in
section~\ref{ss:cofibrant-resolutions}, let $Q$ be a finite graded
quiver, $\hat{kQ}$ the complete path algebra, $\mathfrak{m}$ the
two-sided ideal of $\hat{kQ}$ generated by the arrows of $Q$ and
$A=(\hat{kQ}, d)$ a topological dg algebra whose differential takes
each arrow of $Q$ to an element of $\mathfrak{m}$. Let $\ch_{fd}(A)$
be the full subcategory of $\ch(A)$ formed by the dg modules of finite
total dimension and let $\ac_{fd}(A)$ be its full subcategory formed
by the acyclic dg modules of finite total dimension. Let $\cd_0(A)$ be
the localizing subcategory of $\cd(A)$ generated by $\cd_{fd}(A)$.
Thus, the subcategory $\cd_0(A)$ is the closure of $\cd_{fd}(A)$ under
arbitrary coproducts and finite extensions.

\begin{theorem} \label{thm:findim-compact}
Assume that all arrows of $Q$ are of degree $\leq 0$.
\begin{itemize}
\item[a)] The objects of $\cd_{fd}(A)$ are compact in
$\cd(A)$ and the canonical functor
\[
\ch_{fd}(A)/\ac_{fd}(A) \to \cd_{fd}(A)
\]
is an equivalence.
\item[b)] For each perfect dg module $P$ and each dg module $M$
belonging to $\cd_0(A)$, the canonical map
\[
\colim (\cd A)(P\lten_{A} A/\m^n, M) \to (\cd A)(P, M)
\]
is bijective.
\end{itemize}
\end{theorem}

\begin{proof} a) Let us first show that both categories, $\cd_{fd}(A)$ and
$\hfd{A}$ are generated by the simple modules $S_i$. Indeed, if $M$
is a dg module, then, since $A$ is concentrated in negative degrees,
we have an exact sequence of dg modules
\[
0 \to \tau_{\leq 0} M \to M \to \tau_{>0}(M) \to 0.
\]
If $M$ belongs to $\hfd{A}$ respectively $\cd_{fd}(A)$, this
sequence yields a triangle in the respective triangulated category.
Thus, it is enough to check that an object $M$ whose only non zero
component is in degree~$0$ belongs to the triangulated category
generated by the $S_i$. This is easy by induction on the length of
the zeroth component as a module over the semilocal algebra
$H^0(A)$. Since the $S_i$ are compact, it follows that $\cd_{fd}(A)$
consists of compact objects. Moreover, in order to show the
equivalence claimed in a), it is enough to show that the canonical
functor is fully faithful.  For this, it suffices to show that the
canonical map
\[
(\hfd{A})(S_i, M) \to (\cd_{fd} A)(S_i, M)
\]
is bijective for each vertex $i$ of $Q$ and for each dg module $M$
of finite total dimension. Let us first show that it is surjective.
Let $P_i \to S_i$ be the canonical projection and $R$ its kernel.
Let $j$ be the inclusion of $R$ in $P_i$. By
section~\ref{ss:cofibrant-resolutions}, the cone $C$ over $j$ is a
cofibrant resolution of $S_i$. Since $M$ is of finite total
dimension, it is annihilated by some power $\m^n$ of the ideal $\m$.
Let $f: S_i \to M$ be a morphism of $\cd_{fd}(A)$. It is represented
by a morphism of dg modules $g: C \to M$. Since $C$ is the cone over
$j$, the datum of $g$ is equivalent to that of a morphism of dg
modules $g_0: P_i \to M$ and a morphism of graded modules $g_1: R
\to M$ homogeneous of degree $-1$ and such that $g_0 \circ j =
d(g_1)$. Since $M$ is annihilated by $\m^n$, the morphism $g_0$
factors canonically through the projection $P_i \to P_i/ P_i
\m^{n+1}$ and the morphism $g_1$ through the projection $R \to
R/R\m^n$. We choose these exponents because we have
\[
P_i \m^{n+1} \cap R = R\m^n.
\]
Indeed, the graded module $P_i \m^{n+1}$ has a topological
basis consisting of the paths of length $\geq n+1$ ending
in $i$, the graded module $R$ has a topological basis
consisting of the paths of length $\geq 1$ ending
in $i$ and so $R\m^n$ also has as topological basis
the set of paths of length $\geq n+1$ ending in $i$.
It follows that we have an exact sequence of dg modules
\[
0 \to R/R\m^n \to P_i/P_i\m^{n+1} \to S_i \to 0
\]
and the cone $C'$ over the map
\[
R/R\m^n \to P_i/P_i\m^{n+1}
\]
is still quasi-isomorphic to $S_i$. By construction, the map $g: C
\to M$ factors through the canonical projection $C \to C'$ and $C'$
is of finite total dimension. So the given morphism $f$ is
represented by the fraction
\[
\xymatrix{S_i &  C' \ar[l] \ar[r] &   M}
\]
and hence is the image of a morphism of $\hfd{A}$. To
show injectivity of the canonical map, we first
observe that the functor
\[
\hfd{A} \to \cd_{fd}(A)
\]
detects isomorphisms. Indeed, if a morphism $f: L \to M$
becomes invertible after applying this functor,
then its cone is acyclic and so $f$ was already
invertible. Now we show that the map
\[
(\hfd{A})(S_i, M) \to (\cd_{fd} A)(S_i, M)
\]
is injective. Indeed, if $f$ is a morphism in its
kernel, then in the triangle
\[
\xymatrix{\Sigma^{-1} Y \ar[r]^\eps & S_i \ar[r]^f & M \ar[r] & Y} \ko
\]
the morphism $\eps$ admits a section $s$ in $\cd_{fd}(A)$. This
section lifts to a morphism $\tilde{s}$ by the surjectivity we have
already shown. Then $\eps \tilde{s}$ is invertible since it becomes
invertible in $\cd_{fd}(A)$ and so $\eps$ is a retraction and $f$
vanishes.

b) Both sides are homological functors in $M$ which commute
with arbitrary coproducts. So it is enough to show the
claim for $M$ in $\cd_{fd}(A)$. By a), we may assume that
$M$ belongs to $\ch_{fd}(A)$. We may also assume that $P$
is strictly perfect. Let us show that the map is surjective.
A morphism from $P$ to $M$ in $\cd(A)$ is represented by
a morphism of dg modules $f: P\to M$. Since $M$ is of finite
total dimension, it is annihilated by some power $\m^n$
and so $f$ factors through the projection $P \to P/P\m^n$,
which shows that $f$ is in the image of the map.
Let us now show injectivity. Let $f: P/P\m^n \to M$ be
a morphism of $\cd A$ such that the composition
\[
P \to P/P\m^n \to M
\]
vanishes in $\cd A$. By the calculus of left fractions,
there is a quasi-isomorphism $M \to M'$ such that
the morphism $f$ is represented by a fraction
\[
\xymatrix{ P/P\m^n \ar[r]^-{f'} & M' & M \ar[l] } \ko
\]
where $f'$ is a morphism in $\ch_{fd}(A)$. By forming the shifted
cone over $f'$ we obtain an exact sequence of dg modules
\[
 0 \to \Sigma^{-1} M' \to E \to P/P\m^n \to 0
\]
which splits as a sequence of graded $A$-modules and whose terms all
have finite total dimension. Since the composition $P \to P/P\m^n
\to M'$ vanishes in $\cd A$, the canonical surjection $P \to
P/P\m^n$ factors through $E$ in the category of dg modules:
\[
P \to E \to P/P\m^n.
\]
Since $E$ is of finite total dimension, the map $P\to E$ factors
through the canonical projection $P \to P/P\m^N$ for some $N\gg 0$.
The composition
\[
P/P\m^N \to E \to P/P\m^n
\]
is the canonical projection since its composition with $P \to
P/P\m^N$ is the canonical projection $P\to P/P\m^n$ and $P\to
P/P\m^n$ is surjective. The composition
\[
P/P\m^N \to E \to P/P\m^n \to M'
\]
vanishes in $\cd_{fd}(A)$ since it contains two
successive morphisms of a triangle. Thus the
composition of the given map $f: P/P\m^n \to M$
with the projection $P/P\m^N \to P/P\m^n$
vanishes in $\cd_{fd}(A)$, which implies that
the given map vanishes in the passage to
the colimit.
\end{proof}

\subsection{The Nakayama functor} \label{ss:Nakayama-functor}
We keep the notations and assumptions of
section~\ref{ss:finite-dimensional-dg-modules}.
We will construct a fully faithful
functor from $\per(A)$ to $\cd_0(A)$. Let $P$ be in $\per(A)$. Then
the functor
\[
D\Hom_{\cd(A)}(P,?) : \cd_0(A)\op \to \Mod k
\]
is cohomological and takes coproducts of $\cd_0(A)$ to products
because $P$ is compact. Now the triangulated category $\cd_0(A)$ is
compactly generated by $\cd_{fd}(A)$. So by the Brown
representability theorem, we have an object $\nu P$ in $\cd_0(A)$
and a canonical isomorphism
\[
D\Hom_{\cd(A)}(P, L) = \Hom_{\cd(A)}(L, \nu P)
\]
functorial in the object $L$ of $\cd_0(A)$. As in the case of the
Serre functor (\confer \cite{BondalKapranov89},
\cite{VandenBergh08}), one shows that the assignment $P \mapsto \nu
P$ underlies a canonical triangle functor from $\per(A)$ to
$\cd_0(A)$. We call it the {\em Nakayama functor}. It takes a
strictly perfect object $P$ to the colimit $\colim D(P/P\m^n)$.

\begin{proposition}
\begin{itemize}
\item[a)] The Nakayama functor $\nu: \per(A) \to \cd_0(A)$
is fully faithful.
Its image is formed by the objects $M$ such that for each object
$F$ in $\cd_{fd}(A)$, the sum
\[
\sum_{i\in\Z} \dim \Hom(F,\Sigma^i M)
\]
is finite.
\item[b)] The functor
\[
\Phi: \per(A) \to \Mod(\cd_{fd}(A)\op)
\]
taking an object $P$ to the restriction of $\Hom(P,?)$ to $\cd_{fd}(A)$
is fully faithful. 
\end{itemize}
\end{proposition}

\begin{proof} a) If $P$ and $Q$ are perfect objects, we have
\[
\Hom(\nu P , \nu Q) = D\Hom(Q, \nu P) = D \colim \Hom(Q\lten_A
A/\m^n , \nu P) \ko
\]
where we have used part b) of Theorem~\ref{thm:findim-compact}. Now
by definition of $\nu P$, we have
\[
D \colim \Hom(Q\lten_A A/\m^n , \nu P)  = D \colim D\Hom(P, Q
\lten_A A/\m^n)
\]
and the last space identifies with
\[
\lim DD \Hom(P,  Q \lten_A A/\m^n) = \lim \Hom(P, Q\lten_A A/\m^n) =
\Hom(P,Q).
\]
Thus, the functor $\nu$ is fully faithful. 

b) The given functor takes its values in the full subcategory
$\cu$ of left modules on $\cd_{fd}(A)$ which take values in
the category of finite-dimensional vector spaces. The duality
functor $D : \cu \to \Mod(\cd_{fd}(A))$ is fully faithful
and the composition $D\circ \Phi$ is isomorphic to the
composition of the Yoneda embedding with $\nu$. By part a),
the composition $D\circ \Phi$ and therefore $\Phi$ are
fully faithful.
\end{proof}

\section{The derived equivalence}\label{S:the-equivalence}

\subsection{The main theorem} \label{ss:main-theorem}
Let $(Q,W)$ be a quiver with potential and $i$ a fixed vertex of $Q$. We assume (c1) (c2) and (c3) as in Section~\ref{S:mutation}.
Write $\tilde{\mu}_i(Q,W)=(Q',W')$.
Let $\Gamma=\hat\Gamma(Q,W)$ and $\Gamma'=\hat{\Gamma}(Q',W')$ be the complete Ginzburg dg algebras associated
to $(Q,W)$ and $(Q',W')$ respectively, \confer Section~\ref{ss:ginzburgalg}.

For a vertex $j$ of $Q$, let $P_j=e_j\Gamma$ and $P_j'=e_j\Gamma'$.

\begin{theorem}\label{T:mainthm}
\begin{itemize}
 \item[a)]
 There is a triangle equivalence
\[F:\mathcal{D}(\Gamma')\longrightarrow\mathcal{D}(\Gamma),\]
which sends the $P'_j$ to $P_j$ for $j\neq i$ and to the cone $T_i$
over the morphism
\[P_i\longrightarrow \bigoplus_{\alpha}P_{t(\alpha)}\]
for $i=j$, where we have a summand $P_{t(\alpha)}$ for each arrow
$\alpha$ of $Q$ with source $i$ and the corresponding component of
the map is the left multiplication by $\alpha$. The functor $F$
restricts to triangle equivalences from $\per(\Gamma')$ to
$\per(\Gamma)$ and from $\mathcal{D}_{fd}(\Gamma')$ to
$\mathcal{D}_{fd}(\Gamma)$.

\item[b)] Let $\Gamma_{\mathrm{red}}$ respectively $\Gamma'_{\mathrm{red}}$ be
the complete Ginzburg dg algebra associated with the reduction of
$(Q,W)$ respectively the reduction $\mu_i(Q,W)$ of
$\tilde{\mu}_i(Q,W)$. The functor $F$ yields a triangle equivalence
\[F_{\mathrm{red}}:\mathcal{D}(\Gamma'_{\mathrm{red}})\longrightarrow\mathcal{D}(\Gamma_{\mathrm{red}}),\]
which restricts to triangle equivalences from
$\per(\Gamma'_{\mathrm{red}})$ to $\per(\Gamma_{\mathrm{red}})$ and
from $\mathcal{D}_{fd}(\Gamma'_{\mathrm{red}})$ to
$\mathcal{D}_{fd}(\Gamma_{\mathrm{red}})$.
\end{itemize}
\end{theorem}

\begin{proof}
a) We will prove this in the following two subsections: we will
first construct a $\Gamma'$-$\Gamma$-bimodule $T$, and then we will
show that the left derived tensor functor and the right derived Hom
functor associated to $T$ are quasi-inverse equivalences.

b) This follows from a) and Lemma~\ref{L:reduction}, since
quasi-isomorphisms of dg algebras induce triangle equivalences in
their derived categories.
\end{proof}

\subsection{A dg $\Gamma'$-$\Gamma$-bimodule}\label{ss:bimodule}
First let us analyze the definition of $T_i$. For an arrow $\alpha$
of $Q$ starting at $i$, let $P_{\alpha}$ be a copy of
$P_{t(\alpha)}$. We denote the element $e_{t(\alpha)}$ of this copy
by $e_{\alpha}$. Let $P_{\Sigma i}$ be a copy of $\Sigma P_{i}$. We
denote the element $e_i$ of this copy by $e_{\Sigma i}$. Let $T_{i}$
be the mapping cone of the canonical inclusion
\begin{eqnarray*}P_{i}&\stackrel{\iota}{\longrightarrow}&
\bigoplus_{\alpha\in Q_1:s(\alpha)=i}P_{\alpha}\\
a&\mapsto& \sum_{\alpha\in Q_1:s(\alpha)=i} e_{\alpha}\alpha a.
\end{eqnarray*}
Therefore, as a graded space, $T_i$ has the decomposition
\[T_{i}=P_{\Sigma
i}\oplus\bigoplus_{\alpha\in Q_1:s(\alpha)=i}P_{\alpha},\] and the
differential is given by
\[d_{T_{i}}(e_{\Sigma
i}a_{i}+\sum_{\alpha\in Q_1:s(\alpha)=i}e_{\alpha}a_{\alpha})
=-e_{\Sigma i}d_{\Gamma}(a_{i})+\sum_{\alpha\in
Q_1:s(\alpha)=i}e_{\alpha}(d_{\Gamma}(a_{\alpha})+\alpha a_{i}).\]
Note that each $P_{\alpha}$ is a dg submodule of $T_{i}$.

In this and next subsection, $T_j$ will be $P_j$ for a vertex
$j$ of $Q$ different from $i$. Let $T$ be the direct sum of the
$T_j$, where $j$ runs over all vertices of $Q$, \ie
\[T=\bigoplus_{j\in Q_0}T_j.\]
In the following, for each arrow $\rho'$ of $\tilde{Q'}$, we will
construct  a morphism $f_{\rho'}$ in
$\Lambda=\mathcal{H}om_{\Gamma}(T,T)$ and show that the assignment
$\rho' \to f_{\rho'}$ extends to a morphism of dg algebras from
$\Gamma'$ to $\Lambda$.
\bigskip

For a vertex $j$ of $Q$, let $f_{j}:T_j\rightarrow T_j$ be the identity map.

For an arrow $\alpha$ of $Q$ starting at $i$, we define the morphism
$f_{\alpha^{\star}}:T_{t(\alpha)}\rightarrow T_{i}$ of degree~$0$ as
the composition of the isomorphism $T_{t(\alpha)}\cong P_{\alpha}$
with the canonical embedding $P_{\alpha}\hookrightarrow T_{i}$, \ie
\[f_{\alpha^{\star}}:T_{t(\alpha)}\longrightarrow T_i,a\mapsto e_{\alpha}a,\] and define the morphism
$f_{\alpha^{\star *}}:T_{i}\rightarrow T_{t(\alpha)}$ of degree $-1$
by
\[f_{\alpha^{\star *}}(e_{\Sigma i}a_{i}+\sum_{\rho\in Q_1:s(\rho)=i}e_{\rho}a_{\rho})=-\alpha t_{i} a_{i}-\sum_{\rho\in Q_1:s(\rho)=i}\alpha\rho^{*}a_{\rho}.\]

For an arrow $\beta$ of $Q$ ending at $i$, we define the morphism
$f_{\beta^{\star}}:T_{i}\rightarrow T_{s(\beta)}$ of degree~$0$ by
\[f_{\beta^{\star}}(e_{\Sigma i}a_{i}+\sum_{\rho\in Q_1:s(\rho)=i}e_{\rho}a_{\rho})=-\beta^{*}a_{i}-\sum_{\rho\in Q_1:s(\rho)=i}(\partial_{\rho\beta}W)a_{\rho},\]
and define the morphism $f_{\beta^{\star
*}}:T_{s(\beta)}\rightarrow T_{i}$ of degree $-1$ as the composition
of the morphism $T_{s(\beta)}\rightarrow P_{\Sigma i},~a\mapsto
e_{\Sigma i}\beta a$ and the canonical embedding $P_{\Sigma
i}\hookrightarrow T_{i}$, \ie \[f_{\beta^{\star *}}:T_{s(\beta)}\longrightarrow T_i, a\mapsto e_{\Sigma
i}\beta a.\]

For a pair of arrows $\alpha$, $\beta$ of $Q$ such that $\alpha$ starts at $i$ and $\beta$ ends at $i$, we define
\[f_{[\alpha\beta]}: T_{s(\beta)}\longrightarrow T_{t(\alpha)}, a\mapsto
\alpha\beta a,\] and
\[f_{[\alpha\beta]^{*}}=0:T_{t(\alpha)}\longrightarrow T_{s(\beta)}.\]

For an arrow $\gamma$ of $Q$ not incident to $i$, we denote by
$f_{\gamma}$ and $f_{\gamma^{*}}$ respectively the
left-multiplication by $\gamma$ from $T_{s(\gamma)}$ to
$T_{t(\gamma)}$ and the left-multiplication by $\gamma^{*}$ from
$T_{t(\gamma)}$ to $T_{s(\gamma)}$, \ie
\[f_{\gamma}:T_{s(\gamma)}\longrightarrow T_{t(\gamma)},a\mapsto \gamma a,\]
\[f_{\gamma^*}:T_{t(\gamma)}\longrightarrow T_{s(\gamma)},a\mapsto \gamma^* a.\]

For a vertex $j$ of $Q$ different from $i$, we define $f_{t'_{j}}$
as the left-multiplication by $t_j$ from $T_{j}$ to itself, \ie
\[f_{t'_j}:T_j\longrightarrow T_j, a\mapsto t_j a.\] It is a
morphism of degree $-2$. We define $f_{t'_{i}}$ as the
$\Gamma$-linear morphism of degree $-2$ from $T_{i}$ to itself given by
\[f_{t'_i}(e_{\Sigma i}a_{i}+\sum_{\rho\in Q_1:s(\rho)=i}e_{\rho}a_{\rho})=-e_{\Sigma i}(t_i a_i+\sum_{\rho\in Q_1:s(\rho)=i}\rho^*a_{\rho}).\]

We extend all the morphisms defined above trivially to morphisms
from $T$ to itself.

\begin{proposition} The correspondence taking $e_j$ to $f_{j}$ for all vertices $j$ of $Q$
and taking an arrow $\rho'$ of $\tilde{Q}'$ to $f_{\rho'}$
 extends to a homomorphism of dg algebras
from $\Gamma'$ to $\Lambda=\mathcal{H}om_{\Gamma}(T,T)$. In this way, $T$ becomes a left dg $\Gamma'$-module.
\end{proposition}
\begin{proof} Recall that $\Lambda=\mathcal{H}om_{\Gamma}(T,T)$ is
the dg endomorphism algebra of the dg $\Gamma$-module $T$. The
$\mathfrak{m}$-adic topology of $\Gamma$ induces an
$\mathfrak{m}$-adic topology of $\Lambda$: If we let
\[\mathfrak{n}=\{f\in\Lambda|\im(f)\in T\mathfrak{m}\},\]
then the powers of $\mathfrak{n}$ form a basis of open
neighbourhoods of $0$. The topological algebra $\Lambda$ is complete
with respect to this topology. Note that the image $f_{a^{\star}}$
of the arrow $a^{\star}$ does not lie in $\mathfrak{n}$ but that the
composition of the images of any two arrows of $\tilde{Q}'$ does lie in
$\mathfrak{n}$. Thus the correspondence defined in the proposition
uniquely extends  to a continuous algebra homomorphism. We denote by
$f_{p}$ the image of an element $p$ in $\Gamma'$. It remains to
check that this map commutes with the differentials on generators
$\rho'$ of $\Gamma'$. Namely, we need to show
\begin{eqnarray*}
d_{\Lambda}(f_{\rho'})~~ &=& 0\text{ for each arrow }\rho'\text{ of }Q',\\
d_{\Lambda}(f_{\rho'^*}) &=& f_{\del_{\rho'} W'}\text{ for each arrow }\rho'\text{ of }Q',\\
d_{\Lambda}(f_{t'_j})~~ &=& f_j (\sum_{\rho'\in Q'_1}
[f_{\rho'},f_{\rho'^*}]) f_j\text{ for each vertex } j\text{ of }Q'.
\end{eqnarray*}
We will prove the first two equalities in the following
Lemmas~\ref{L:dalpha}--\ref{L:dgamma} and the last equality in
Lemma~\ref{L:dr}.
\end{proof}

\begin{lemma}\label{L:dalpha} Let $\alpha$ be an arrow of $Q$ starting at $i$. Then
\[d_{\Lambda}(f_{\alpha^{\star}})=0,\qquad d_{\Lambda}(f_{\alpha^{\star *}})=f_{\partial_{\alpha^{\star}}W'}.\]
\end{lemma}
\begin{proof} It suffices to check the equalities
on the generator $e_{t(\alpha)}$ of $T_{t(\alpha)}$ and on the
generators $e_{\Sigma i}$ and $e_{a}$, $a:i\rightarrow t(a)$ in $Q$,
of $T_{i}$. We have
\begin{eqnarray*}d_{\Lambda}(f_{\alpha^{\star}})(e_{t(\alpha)})&=&d_{T_{i}}(f_{\alpha^{\star}}(e_{t(\alpha)}))-(-1)^{0}f_{\alpha^{\star}}(d_{T_{t(\alpha)}}(e_{t(\alpha)}))\\
&=&d_{T_{i}}(e_{\alpha})\\
&=&0.
\end{eqnarray*}
Therefore $d_{\Lambda}(f_{\alpha^{\star}})=0$.

Let $a$ be any arrow of $Q$ starting from $i$. We have
\begin{eqnarray*}
d_{\Lambda}(f_{\alpha^{\star
*}})(e_{a})&=&d_{T_{t(\alpha)}}(f_{\alpha^{\star
*}}(e_{a}))-(-1)^{-1}f_{\alpha^{\star *}}(d_{T_{i}}(e_{a}))\\
&=&d_{T_{t(\alpha)}}(-\alpha a^{*})\\
&=&-\alpha \del_a W,\\
 d_{\Lambda}(f_{\alpha^{\star *}})(e_{\Sigma i})&=&d_{T_{t(\alpha)}}(f_{\alpha^{\star
*}}(e_{\Sigma i}))-(-1)^{-1}f_{\alpha^{\star *}}(d_{T_{i}}(e_{\Sigma i}))\\
&=&d_{T_{t(\alpha)}}(-\alpha t_{i})+f_{\alpha^{\star
*}}(\sum_{a\in Q_1:s(a)=i}e_{a}a)\\
&=&-\alpha(\sum_{\beta\in Q_1:t(\beta)=i}\beta\beta^*-\sum_{a\in Q_1:s(a)=i}a^*a)+\sum_{a\in Q_1:s(a)=i}(-\alpha a^{*})a\\
&=&-\alpha(\sum_{\beta\in Q_1:t(\beta)=i}\beta\beta^{*}).
\end{eqnarray*}

It follows from the definition of $W'$ that
\[\partial_{\alpha^{\star}}W'=\partial_{\alpha^{\star}}(W_1'+W_2')=\partial_{\alpha^{\star}}W_{2}'=\sum_{\beta\in Q_1:t(\beta)=i}[\alpha\beta]\beta^{\star}.\]
Thus
\begin{eqnarray*}
f_{\partial_{\alpha^{\star}}W'}(e_{a})&=&\sum_{\beta\in
Q_1:t(\beta)=i}f_{[\alpha\beta]}f_{\beta^{\star}}(e_{a})\\
&=&\sum_{\beta\in
Q_1:t(\beta)=i}f_{[\alpha\beta]}(-\partial_{a\beta}W)\\
&=&\sum_{\beta\in
Q_1:t(\beta)=i}(-\alpha\beta\partial_{a\beta}W)\\
&=&-\alpha\sum_{\beta\in
Q_1:t(\beta)=i}(\beta\partial_{a\beta}W)\\
&=&-\alpha\partial_{a}W,\\
 f_{\partial_{\alpha^{\star}}W'}(e_{\Sigma i})&=&\sum_{\beta\in
Q_1:t(\beta)=i}f_{[\alpha\beta]}f_{\beta^{\star}}(e_{\Sigma i})\\
&=&\sum_{\beta\in
Q_1:t(\beta)=i}f_{[\alpha\beta]}(-\beta^{*})\\
&=&-\alpha(\sum_{\beta\in Q_1:t(\beta)=i}\beta\beta^{*}).
\end{eqnarray*}
Therefore $d_{\Lambda}(f_{\alpha^{\star
*}})=f_{\partial_{\alpha^{\star}}W'}$.
\end{proof}

\begin{lemma}\label{L:dbeta} Let $\beta$ be an arrow of $Q$ ending at $i$. Then
\[d_{\Lambda}(f_{\beta^{\star}})=0,\qquad d_{\Lambda}(f_{\beta^{\star *}})=f_{\partial_{\beta^{\star}}W'}.\]
\end{lemma}
\begin{proof}
For an arrow $\alpha$ of $Q$ starting at $i$, we have
\begin{eqnarray*}
d_{\Lambda}(f_{\beta^{\star}})(e_{\alpha})&=&d_{T_{s(\beta)}}(f_{\beta^{\star}}(e_{\alpha}))-(-1)^{0}f_{\beta^{\star}}(d_{T_{k}}(e_{\alpha}))\\
&=&d_{T_{s(\beta)}}(-\partial_{\alpha\beta}W)\\
&=&0,\\
d_{\Lambda}(f_{\beta^{\star}})(e_{\Sigma i})&=&d_{T_{s(\beta)}}(f_{\beta^{\star}}(e_{\Sigma i}))-(-1)^{0}f_{\beta^{\star}}(d_{T_{i}}(e_{\Sigma i}))\\
&=&d_{T_{s(\beta)}}(-\beta^{*})-f_{\beta^{\star}}(\sum_{\alpha\in
Q_1:s(\alpha)=i}e_{\alpha}\alpha)\\
&=&d_{T_{s(\beta)}}(-\beta^{*})-\sum_{\alpha\in
Q_1:s(\alpha)=i}(-(\partial_{\alpha\beta}W)\alpha )\\
&=&-\del_{\beta}W+ \sum_{\alpha\in Q_1:s(\alpha)=i}(\partial_{\alpha\beta}W)\alpha\\
&=&0.
\end{eqnarray*}
This shows the first equality.
The second equality holds because
\begin{eqnarray*}
d_{\Lambda}(f_{\beta^{\star
*}})(e_{s(\beta)})&=&d_{T_{i}}(f_{\beta^{\star
*}}(e_{s(\beta)}))-(-1)^{-1}f_{\beta^{\star *}}(d_{T_{s(\beta)}}(e_{s(\beta)}))\\
&=&d_{T_{i}}(e_{\Sigma i} \beta)\\
&=&\sum_{\alpha\in Q_1:s(\alpha)=i}e_{\alpha}\alpha\beta,\\
f_{\partial_{\beta^{\star}}W'}(e_{s(\beta)})&=&f_{\partial_{\beta^{\star}}W_2'}(e_{s(\beta)})\\
&=&\sum_{\alpha\in Q_1:s(\alpha)=i}f_{\alpha^{\star}}f_{[\alpha\beta]}(e_{s(\beta)})\\
&=&\sum_{\alpha\in Q_1:s(\alpha)=i}f_{\alpha^{\star}}(\alpha\beta)\\
&=&\sum_{\alpha\in Q_1:s(\alpha)=i}e_{\alpha}\alpha\beta.
\end{eqnarray*}
\end{proof}

\begin{lemma}\label{L:dalphabeta} Let $\alpha$, $\beta$ be arrows of $Q$ starting and
ending at $i$ respectively. Then
\[d_{\Lambda}(f_{[\alpha\beta]})=0,\qquad d_{\Lambda}(f_{[\alpha\beta]^{*}})=f_{\partial_{[\alpha\beta]}W'}.\]
\end{lemma}

\begin{proof}

The first equality follows from
\begin{eqnarray*}
d_{\Lambda}(f_{[\alpha\beta]})(e_{s(\beta)})&=&d_{T_{t(\alpha)}}(f_{[\alpha\beta]}(e_{s(\beta)}))-(-1)^{0}f_{[\alpha\beta]}(d_{T_{s(\beta)}}(e_{s(\beta)}))\\
&=&d_{T_{t(\alpha)}}(\alpha\beta)\\
&=&0.
\end{eqnarray*}

Since $f_{[\alpha\beta]^{*}}$ is the zero morphism, its differential
$d_{\Lambda}(f_{[\alpha\beta]^{*}})$ is zero as well. On the other hand,
\begin{eqnarray*}
f_{\partial_{[\alpha\beta]}W'}(e_{t(\alpha)})&=&f_{\partial_{[\alpha\beta]}W'_1+\beta^{\star}\alpha^{\star}}(e_{t(\alpha)})\\
&=&f_{\partial_{[\alpha\beta]}W'_1}(e_{t(\alpha)})+f_{\beta^{\star}\alpha^{\star}}(e_{t(\alpha)})\\
&=&f_{[\partial_{\alpha\beta}W]}(e_{t(\alpha)})+(-\partial_{\alpha\beta}W)\\
&=&0.
\end{eqnarray*}
Here we applied the following observation: Let $p$ be a path of $Q$
neither starting or ending at $i$ and let $[p]$ be the path of $Q'$
obtained from $p$ by replacing all $ab$ by $[ab]$, where $a$, $b$
are a pair of arrows of $Q$ such that $a$ starts at $i$ and $b$ ends
at $i$. Then the map $f_{[p]}$ is the left multiplication by $p$.
Therefore
\[f_{\partial_{[\alpha\beta]}W'}=0=d_{\Lambda}(f_{[\alpha\beta]^{*}}).\]
\end{proof}

\begin{lemma}\label{L:dgamma} Let $\gamma$ be an arrow of $Q$ not starting or ending at
$i$. Then
\[d_{\Lambda}(f_{\gamma})=0,\qquad d_{\Lambda}(f_{\gamma^{*}})=f_{\partial_{\gamma}W'}.\]
\end{lemma}
\begin{proof}

The first equality follows from
\begin{eqnarray*}
d_{\Lambda}(f_{\gamma})(e_{s(\gamma)})&=&d_{T_{t(\gamma)}}(f_{\gamma}(e_{s(\gamma)}))-(-1)^{0}f_{\gamma}(d_{T_{s(\gamma)}}(e_{s(\gamma)}))\\
&=&d_{T_{t(\gamma)}}(\gamma)\\
&=&0.
\end{eqnarray*}

For the second, we have
\begin{eqnarray*}
d_{\Lambda}(f_{\gamma^{*}})(e_{t(\gamma)})&=&d_{T_{s(\gamma)}}(f_{\gamma^{*}}(e_{t(\gamma)}))-(-1)^{-1}f_{\gamma^{*}}(d_{T_{t(\gamma)}}(e_{t(\gamma)}))\\
&=&d_{T_{s(\gamma)}}(\gamma^{*})\\
&=&\partial_{\gamma}W,\\
f_{\partial_{\gamma}W'}(e_{t(\gamma)})&=&f_{\partial_{\gamma}W'_{1}}(e_{t(\gamma)})\\
&=&f_{[\del_{\gamma}W]}(e_{t(\gamma)})\\
&=&\partial_{\gamma}W.
\end{eqnarray*}
Therefore $d_{\Lambda}(f_{\gamma^{*}})=f_{\partial_{\gamma}W'}$.

\end{proof}

\begin{lemma}\label{L:dr}
Let $j\in Q_0$. Then
\[d_{\Lambda}(f_{t'_{j}})=\sum_{\rho'\in Q_1':t(\rho')=j}f_{\rho'}f_{\rho'^{*}}-\sum_{\rho'\in Q_1':s(\rho')=j}f_{\rho'^{*}}f_{\rho'}\]
\end{lemma}
\begin{proof} We divide the proof into two cases.

\noindent Case $j\neq i$. We have
\begin{eqnarray*}
d_{\Lambda}(f_{t'_{j}})(e_j)&=&d_{T_{j}}(f_{t'_{j}}(e_j))-(-1)^{-2}f_{t'_{j}}(d_{T_{j}}(e_j))\\
&=&d_{T_{j}}(t_{j}),\\
\sum_{\rho'\in Q_1':t(\rho')=j}f_{\rho'}f_{\rho'^{*}}(e_j)
&=&(\sum_{\beta:j\rightarrow i\text{ in }
Q_1}f_{\beta^{\star}}f_{\beta^{\star
*}}+\sum_{\gamma:s(\gamma)\rightarrow j\text{ in
}Q_1,s(\gamma)\neq i}f_{\gamma}f_{\gamma^{*}})(e_j)\\
&=&\sum_{\beta:j\rightarrow i\text{ in } Q_1}(-\beta^{*}\beta)
+\sum_{\gamma:s(\gamma)\rightarrow j\text{ in }Q_1,s(\gamma)\neq
i}\gamma\gamma^{*}\\
(-\sum_{\rho'\in Q_1':s(\rho')=j}f_{\rho'^{*}}f_{\rho'})(e_j)
&=&(-\sum_{\alpha:i\rightarrow j\text{ in } Q_1}f_{\alpha^{\star
*}}f_{\alpha^{\star}}
-\sum_{\gamma:j\rightarrow t(\gamma)\neq i\text{
in } Q_1}f_{\gamma^{*}}f_{\gamma})(e_j)\\
&=&-\sum_{\alpha:i\rightarrow j\text{ in } Q_1}(-\alpha\alpha^{*})
-\sum_{\gamma:j\rightarrow t(\gamma)\neq i\text{
in } Q_1}\gamma^{*}\gamma \\
\text{the sum of the last two}&=&\sum_{\rho:s(\rho)\rightarrow
i\text{ in } Q_1}\rho\rho^{*}-\sum_{\rho:j\rightarrow t(\rho)\text{
in
}Q_1}\rho^{*}\rho\\
&=&d_{T_{j}}(t_{j}).\end{eqnarray*}
Therefore \[d_{\Lambda}(f_{t'_{j}})=\sum_{\rho'\in
Q_1':t(\rho')=j}f_{\rho'}f_{\rho'^{*}}-\sum_{\rho'\in
Q_1':s(\rho')=j}f_{\rho'^{*}}f_{\rho'}.\]

\noindent Case $j=i$. For an arrow $a$ of $Q$ starting at $i$, we have
\begin{eqnarray*}
d_{\Lambda}(f_{t'_{i}})(e_{a})&=&d_{T_{i}}(f_{t'_{i}}(e_{a}))-(-1)^{-2}f_{t'_{i}}(d_{T_{i}}(e_{a}))\\
&=&d_{T_{i}}(-e_{\Sigma i}a^{*})\\
&=&-\sum_{\alpha\in Q_1:s(\alpha)=i}e_{\alpha}\alpha a^{*}+e_{\Sigma i}\partial_{a}W,\\
d_{\Lambda}(f_{t'_{i}})(e_{\Sigma i})
&=&d_{T_{i}}(f_{t'_{i}}(e_{\Sigma i}))-(-1)^{-2}f_{t'_{i}}(d_{T_{i}}(e_{\Sigma i}))\\
&=&d_{T_{i}}(-e_{\Sigma i} t_{i})-f_{t'_{i}}(\sum_{\alpha\in Q_1:s(\alpha)=i}e_{\alpha}\alpha)\\
&=&-\sum_{\alpha\in Q_1:s(\alpha)=i}e_{\alpha}\alpha t_{i} + e_{\Sigma i} d_{\Gamma}(t_k)-\sum_{\alpha\in Q_1:s(\alpha)=i}(-e_{\Sigma i} \alpha^{*})\alpha\\
&=&-\sum_{\alpha\in Q_1:s(\alpha)=i}e_{\alpha}\alpha t_{i} +
\sum_{\beta\in Q_1:t(\beta)=i} e_{\Sigma i}\beta\beta^{*}.
\end{eqnarray*}

On the other hand, \begin{eqnarray*} \sum_{\alpha\in
Q_1:s(\alpha)=i}f_{\alpha^{\star}}f_{\alpha^{\star
*}}(e_{a})
&=&\sum_{\alpha\in Q_1:s(\alpha)=i}(-e_{\alpha}\alpha a^{*})\\
&=&-\sum_{\alpha\in Q_1:s(\alpha)=i}e_{\alpha}\alpha a^{*}\\
(-\sum_{\beta\in Q_1:t(\beta)=i}f_{\beta^{\star
*}}f_{\beta^{\star}})(e_{a})
&=&-\sum_{\beta\in Q_1:t(\beta)=i}(-e_{\Sigma i}\beta\partial_{a\beta}W)\\
&=&e_{\Sigma i}\partial_{a}W\\[5pt]
\sum_{\alpha\in Q_1:s(\alpha)=i}f_{\alpha^{\star}}f_{\alpha^{\star
*}}(e_{\Sigma i})
&=&\sum_{\alpha\in Q_1:s(\alpha)=i}(-e_{\alpha}\alpha t_{i})\\
&=&-\sum_{\alpha\in Q_1:s(\alpha)=i}e_{\alpha}\alpha t_{i}\\
(-\sum_{\beta\in Q_1:t(\beta)=i}f_{\beta^{\star
*}}f_{\beta^{\star}})(e_{\Sigma i}) &=&-\sum_{\beta\in Q_1:t(\beta)=i}(-e_{\Sigma i} \beta^{*})\\
&=&\sum_{\beta\in Q_1:t(\beta)=i}e_{\Sigma i} \beta\beta^{*}.
\end{eqnarray*}
Therefore
\begin{eqnarray*}d_{\Lambda}(f_{t'_{i}})&=&\sum_{\alpha\in Q_1:s(\alpha)=i}f_{\alpha^{\star}}f_{\alpha^{\star
*}}-\sum_{\beta\in Q_1:t(\beta)=i}f_{\beta^{\star
*}}f_{\beta^{\star}}\\
&=&\sum_{\rho'\in
Q_1':t(\rho')=i}f_{\rho'}f_{\rho'^{*}}-\sum_{\rho'\in
Q_1':s(\rho')=i}f_{\rho'^{*}}f_{\rho'}.
\end{eqnarray*}

\end{proof}

\subsection{Proof of the equivalence}

In the preceding subsection we constructed a dg $\Gamma'$-$\Gamma$-bimodule $T$. Clearly $T$ is cofibrant as a right dg $\Gamma$-module. Consequently we obtain a pair of adjoint triangle functors $F=?\lten_{\Gamma'}T$ and  $G=\mathcal{H}om_{\Gamma}(T,?)$ between the derived categories $\cd(\Gamma')$ and $\cd(\Gamma)$. In this subsection we will prove that they are quasi-inverse equivalences.

We will denote by $S_j$ and $S'_j$ respectively the simple modules
over $\Gamma$ and $\Gamma'$ attached to the vertex $j$ of $Q$. As
shown in the proof of Theorem~\ref{thm:findim-compact} a), they
respectively generate the triangulated categories $\cd_{fd}(\Gamma)$
and $\cd_{fd}(\Gamma')$.

\begin{lemma}\label{L:imageofsimples}
\begin{itemize}
 \item[a)] Let $j$ be a vertex of $Q$. If $j=i$, we have an isomorphism in $\mathcal{D}(\Gamma)$
\[F(S'_i)\cong\Sigma S_i.\]
If $j\neq i$, then $FS'_j$ is isomorphic in $\mathcal{D}(\Gamma)$ to the cone of the canonical map
\[\Sigma^{-1}S_j\longrightarrow\Hom_{\mathcal{D}(\Gamma)}(\Sigma^{-1}S_j,S_i)\ten_k S_i.\]
\item[b)]
For any vertex $j$ of $Q$, we have an isomorphism in $\mathcal{D}(\Gamma')$
\[GF(S'_j)\cong S'_j.\]
\end{itemize}
\end{lemma}
\begin{proof} As shown in Section~\ref{ss:cofibrant-resolutions}, a cofibrant resolution of $S'_j$ is given by the graded vector space
\[\mathbf{p}S'_j=\Sigma^3 P'_j\oplus\bigoplus_{\rho\in Q'_1:s(\rho)=j}\Sigma^2 P'_{t(\rho)}\oplus\bigoplus_{\tau\in Q'_1:t(\tau)=j}\Sigma P'_{s(\tau)}\oplus P'_j,\]
and the differential
\[d_{\mathbf{p}S'_j}=\left(\begin{array}{cccc}
              d_{\Sigma^3 P'_j} & 0 & 0 & 0\\
              \rho & d_{\Sigma^2 P'_{t(\rho)}} & 0 & 0\\
              -\tau^* & -\del_{\rho\tau}W' & d_{\Sigma P'_{s(\tau)}} & 0\\
              t'_j & \rho^* & \tau & d_{P'_j}
             \end{array}\right)
\]
where the obvious summation symbols over $\rho$, $\tau$ and $j$ are omitted.
Therefore
\[F(S'_j)={}\mathbf{p}S'_j\ten_{\Gamma'}T\]
is the dg $\Gamma$-module whose underlying graded vector space is
\[\Sigma^3 T_j\oplus\bigoplus_{\rho\in Q'_1:s(\rho)=j}\Sigma^2 T_{t(\rho)}^{\rho}\oplus\bigoplus_{\tau\in Q'_1:t(\tau)=j}\Sigma T_{s(\tau)}^{\tau}\oplus T_j,\]
where $T_{t(\rho)}^{\rho}$ (repectively, $T_{s(\tau)}^{\tau}$) is the direct summands corresponding to $\rho$ (respectively, $\tau$),
and whose differential is
\[d_{F(S'_j)}=\left(\begin{array}{cccc}
              d_{\Sigma^3 T_j} & 0 & 0 & 0\\
              f_{\rho} & d_{\Sigma^2 T_{t(\rho)}^{\rho}} & 0 & 0\\
              -f_{\tau^*} & -f_{\del_{\rho\tau}W'} & d_{\Sigma T_{s(\tau)}^{\tau}} & 0\\
              f_{t'_j} & f_{\rho^*} & f_{\tau} & d_{T_j}
             \end{array}\right).
\]
Then $F(S'_j)$ is the mapping cone of the morphism
\[\xymatrix{M\ar[rr]^{f=(f_{t'_i},f_{\rho^*},f_{\tau})} && T_j},\]
where $M$ is he dg $\Gamma$-module whose underlying graded vector space is
\[\Sigma^2 T_j\oplus\bigoplus_{\rho\in Q'_1:s(\rho)=j}\Sigma T_{t(\rho)}^{\rho}\oplus\bigoplus_{\tau\in Q'_1:t(\tau)=j}T_{s(\tau)}^{\tau}\]
and whose differential is
\[d_M=\left(\begin{array}{ccc}
              d_{\Sigma^2 T_j} & 0 & 0\\
              -f_{\rho} & d_{\Sigma T_{t(\rho)}^{\rho}} & 0\\
              f_{\tau^*} & f_{\del_{\rho\tau}W'} & d_{T_{s(\tau)}^{\tau}}\\
             \end{array}\right).
\]

We will simplify the description of $FS'_j$ by showing that the
kernel $\ker(f)$ of $f:M\rightarrow T_j$ is contractible and
computing its cokernel. Then we will apply the functor $G$ to the
cokernel. We distinguish two cases: $j=i$ and $j\neq i$.

\noindent Case $j=i$.  Let us compute the effect of $f$ on the
generators of the three families of summands of $M$. Recall that
$T_i$ is generated by $e_{\Sigma i}$ and $e_{\alpha}$,
$\alpha:i\rightarrow t(\alpha)$ in $Q$, and $T_j$, $j\neq i$, is
generated by $e_j$.

(1) The morphism $f_{t'_i}:\Sigma^2 T_i\rightarrow T_i$ sends
$e_{\Sigma i}$ to $-e_{\Sigma i}t_i$, and $e_{\alpha}$
($\alpha:i\rightarrow t(\alpha)$ in $Q$) to $-e_{\Sigma i}\alpha^*$.

(2) For $\rho:i\rightarrow t(\rho)$: in this case $\rho=b^{\star}$
for some $b:s(b)\rightarrow i$ in $Q$. The morphism $f_{b^{\star
*}}:\Sigma T_{s(b)}^{b^{\star}}\rightarrow T_i$ sends
$e_{s(b)}^{b^{\star}}$ to $e_{\Sigma i}b$.

(3) For $\tau:s(\tau)\rightarrow i$: in this case $\tau=a^{\star}$
for some $a:i\rightarrow t(a)$ in $Q$. The morphism
$f_{a^{\star}}:T_{t(a)}^{a^{\star}}\rightarrow T_i$ sends
$e_{t(a)}^{a^{\star}}$ to $e_{a}$.

This description shows that the morphism $f$ is injective. Moreover, the image $\im(f)$ is the dg $\Gamma$-module generated by $e_{\Sigma i}t_i$, $e_{\Sigma i}\alpha^*$, $e_{\Sigma i}b$, and $e_{\alpha}$ (where $\alpha:i\rightarrow t(\alpha)$ and $b:s(b)\rightarrow i$ are arrows in $Q$). Thus the cokernel of $f$ is $\Sigma S_i$. So \[F(S'_i)\cong\Sigma S_i,\] and
\[GF(S'_i)\cong G(\Sigma S_i)=\mathcal{H}om_{\Gamma}(T,\Sigma S_i)=\mathcal{H}om_{\Gamma}(P_{\Sigma i},\Sigma S_i)=S'_i.\]

\noindent Case $j\neq i$.

(1) The morphism $f_{t'_j}: \Sigma^2 T_j\rightarrow T_j$ sends $e_j$ to $t_j$.

(2) For $\rho:j\rightarrow t(\rho)$, we have $\rho=\gamma$ for some $\gamma:j\rightarrow t(\gamma)\neq i$ in $Q$, or $\rho=[ab]$ for some $a:i\rightarrow t(a)$ and $b:j\rightarrow i$ in $Q$, or $\rho=a^{\star}$ for some $a:i\rightarrow j$ in $Q$. In the first case, the morphism $f_{\gamma^*}:\Sigma T_{t(\gamma)}^{\gamma}\rightarrow T_j$ sends $e_{t(\gamma)}^{\gamma}$ to $\gamma^*$. In the second case, the morphism $f_{[ab]^*}:\Sigma T_{t(ab)}^{[ab]}\rightarrow T_i$ is zero.
In the third case, the morphism $f_{a^{\star *}}:\Sigma T_i^{a^{\star}}\rightarrow T_i$ sends $e_{\alpha}^{a^{\star}}$ to $-a\alpha^*$ and sends $e_{\Sigma i}^{a^{\star}}$ to $-at_i$.

(3) For $\tau:s(\tau)\rightarrow j$, we have $\tau=\gamma$ for some $\gamma:i\neq s(\gamma)\rightarrow j$ in $Q$, or $\tau=[ab]$ for some $a:i\rightarrow j$ and $b:s(b)\rightarrow i$ in $Q$, or $\tau=b^{\star}$ for some $b:j\rightarrow i$ in $Q$. In the first case, the morphism $f_{\gamma}:T_{s(\gamma)}^{\gamma}\rightarrow T_j$ sends $e_{s(\gamma)}^{\gamma}$ to $\gamma$. In the second case, the morphism $f_{[ab]}:T_{s(ab)}^{[ab]}\rightarrow T_j$ sends $e_{s(ab)}^{[ab]}$ to $ab$. In the third case, the morphism $f_{b^{\star}}:T_i^{b^{\star}}\rightarrow T_j$ sends $e_{\alpha}^{b^{\star}}$ to $-\del_{\alpha b}W$ and sends $e_{\Sigma i}^{b^{\star}}$ to $-b^*$.

By (2), the kernel $\ker(f)$ of $f$ contains all $\Sigma T_{t(ab)}^{[ab]}$.
For a pair of arrows $a:i\rightarrow t(a)$ and $b:j\rightarrow i$ in $Q$, let $R_{a,b}$ denote the graded $\Gamma$-submodule of $M$ generated by
\[r_{a,b}=((e_{s(\gamma)}^{\gamma}\del_{ab\gamma}W)_{\gamma},(e_{s(a'b')}^{[a'b']}\del_{ab a'b'}W)_{a',b'},e_{a}^{b^{\star}})\]
where $\gamma$ runs over all $\gamma:i\neq s(\gamma)\rightarrow j$ in $Q$ and $a'$ and $b'$ run over all $a':i\rightarrow j$ and $b':s(b')\rightarrow i$ in $Q$. The degree of $r_{a,b}$ is zero, so $R_{a,b}$ is indeed a dg $\Gamma$-submodule of $M$.
It is contained in $\ker(f)$ since
\begin{eqnarray*}f(r_{a,b})&=&\sum_{\gamma}f_{\gamma}(e_{s(\gamma)}^{\gamma}\del_{ab\gamma}W)+\sum_{a',b'}f_{[a'b']}(e_{s(a'b')}^{[a'b']}\del_{aba'b'}W)+f_{b^{\star}}(e_a^{b^{\star}})\\
 &=&\sum_{\gamma}\gamma\del_{a b\gamma}W+\sum_{a',b'}a'b'\del_{a ba'b'}W-\del_{a b}W\\
&=&0.
\end{eqnarray*}
Moreover, it is straightforward to show that the sum of all $R_{a,b}$ is a direct sum and that the underlying graded space of $\ker(f)$ is
\[\bigoplus_{a,b}\Sigma T_{t(ab)}^{[ab]} \oplus\bigoplus_{a,b} R_{a,b},\]
and its differential is
\[d_{\ker(f)}=\left(\begin{array}{cc} d_{\Sigma T_{t(ab)}^{[ab]}} & 0\\
(f_{\del_{[ab]\tau}W'})_{\tau} & d_{R_{a,b}}\end{array}\right),\]
where $\tau$ runs over all arrows $\tau:s(\tau)\rightarrow j$ of
$Q'$. If $\tau=\gamma$ for some $\gamma:i\neq s(\gamma)\rightarrow
j$ in $Q$, then
\[f_{\del_{[ab]\gamma}W'}(e_{t(ab)}^{[ab]})=(e_{s(\gamma)}^{\gamma}\del_{ab
\gamma}W)_{\gamma}.\] If $\tau=[a'b']$ for some $a':i\rightarrow j$
and $b':s(b')\rightarrow i$ in $Q$, then
\[f_{\del_{[ab][a'b']}W'}(e_{t(ab)}^{[ab]})=e_{s(a'b')}^{[a'b']}\del_{ab a'b'}W.\]
If $\tau=b'^{\star}$ for some $b':j\rightarrow i$ in $Q$, then
\[f_{\del_{[ab]b'^{\star}}W'}(e_{t(ab)}^{[ab]})=\delta_{b,b'}e_{a}^{b^{\star}}.\]
Summing them up, we see that the differential of $\ker(f)$ sends $e_{t(ab)}^{[ab]}$ to $r_{a,b}$, and this induces an isomorphism of degree $1$ from $\Sigma T_{t(ab)}^{[ab]}$ to $R_{a,b}$.
In particular, $\ker(f)$ is contractible.

Now $\im(f)$ is generated by $t_j$, $\gamma^*$ ($\gamma:j\rightarrow
t(\gamma)\neq i$), $a\alpha^*$ ($a:i\rightarrow j$ and
$\alpha:i\rightarrow t(\alpha)$ in $Q$), $at_i$ ($a:i\rightarrow j$
in $Q$), $\gamma$ ($\gamma:i\neq s(\gamma)\rightarrow j$ in $Q$),
and $ab$ ($a:i\rightarrow j$ and $b:s(b)\rightarrow i$ in $Q$),
$b^*$ ($b:j\rightarrow i$ in $Q$). Therefore $\cok(f)$ is the vector
space $k\{\bar{e}_j,\bar{a}|a:i\rightarrow j\text{ in } Q\}$
concentrated in degree~$0$ with the obvious $\Gamma$-action. Thus we
have that
\[FS'_j\cong\cok(f)\]
is the universal extension of $S_j$ by $S_i$, or in other words, it is isomorphic to the cone of the canonical map
\[\Sigma^{-1}S_j\longrightarrow\Hom_{\mathcal{D}(\Gamma)}(\Sigma^{-1}S_j,S_i)\ten_k S_i.\]
Further,
\begin{eqnarray*}GF(S'_j)&\cong& G(\cok(f))\\
 &=&\mathcal{H}om_{\Gamma}(T,\cok(f)),\end{eqnarray*}
and so the complex underlying $GFS'_j$ is
\[ \xymatrix{0\ar[r] & k\{g_j, g_a|a:i\rightarrow j\text{ in } Q\}\ar[r]^d &k\{h_{a}|a:i\rightarrow j\text{ in } Q\}\ar[r] & 0,}\]
where $g_j:P_j\rightarrow \cok(f)$ is the canonical projection,
$g_a:P_a\rightarrow\cok(f)$ is a copy of $g_j$, and $h_{a}:\Sigma
P_i\rightarrow \cok(f)$ is of degree $1$ sending $e_{\Sigma i}$ to
$\bar{a}$. The differential $d$ sends $g_a$ to $h_a$. Thus the above
dg $\Gamma'$-module is quasi-isomorphic to $k\{g_j\}$ concentrated
in degree~$0$. From the left $\Gamma'$-action on $T$ we deduce that
$g_je_j=g_j$. Therefore $GF(S'_j)$ is isomorphic to $S'_j$.
\end{proof}

\begin{proposition}\label{P:equivfd}
 The functors $F$ and $G$ induce a pair of quasi-inverse triangle equivalences
\[\xymatrix{\mathcal{D}_{fd}(\Gamma)\ar[r]<.7ex>^G &\mathcal{D}_{fd}(\Gamma')\ar[l]<.7ex>^F}\]
\end{proposition}
\begin{proof}
 Let $\eta:Id\rightarrow GF$ be the unit of the adjoint pair $(F,G)$. Let $j$ be a vertex of $Q$. We would like to show that
\[\eta_{S'_j}:S'_j\rightarrow GFS'_j\]
is invertible. Under the adjunction
\[\Hom_{\mathcal{D}_{fd}(\Gamma)}(FS'_j,FS'_j)\stackrel{\sim}{\longrightarrow}\Hom_{\mathcal{D}_{fd}(\Gamma')}(S'_j,GFS'_j),\]
the morphism $\eta_{S'_j}$ corresponds to the identity of $FS'_j$. So $\eta_{S'_j}$ is nonzero. Since $S'_j$ is isomorphic to $GFS'_j$ (Lemma~\ref{L:imageofsimples}), and $\Hom_{\mathcal{D}_{fd}(\Gamma')}(S'_j,S'_j)$ is one dimensional, it follows that $\eta_{S'_j}$
is invertible. Therefore, so are its shifts $\eta_{\Sigma^p S'_j}=\Sigma^p \eta_{S'_j}$ ($p\in\mathbb{Z}$). Let
\[\xymatrix{X_1\ar[r]^f & X\ar[r]^g & X_2\ar[r]^h& \Sigma X_1}\]
be a triangle in $\mathcal{D}_{fd}(\Gamma')$. Then we obtain a commutative diagram
\[\xymatrix{X_1\ar[r]^f\ar[d]^{\eta_{X_1}} & X\ar[r]^g\ar[d]^{\eta_{X}} & X_2\ar[r]^h\ar[d]^{\eta_{X_2}} &\Sigma X_1 \ar[d]^{\Sigma\eta_{X_1}}\\
GFX_1\ar[r]^{GFf} & GFX\ar[r]^{GFg} & GFX_2\ar[r]^{GFh} &\Sigma
GFX_1.}\] Thus if $\eta_{X_1}$ and $\eta_{X_2}$ are isomorphisms,
then so is $\eta_X$.  Since $\mathcal{D}_{fd}(\Gamma')$ is generated
by all the $S'_j$'s, we deduce that for all objects $X$ of
$\mathcal{D}_{fd}(\Gamma')$, the morphism $\eta_X$ is an
isomorphism. It follows that $F$ is fully faithful. Moreover, by
Lemma~\ref{L:imageofsimples} a), the objects $FS'_j$, $j\in Q_0$,
generate the triangulated category $\mathcal{D}_{fd}(\Gamma)$, since
they generate all the simples $S_j$. Therefore $F$ is an
equivalence, and it follows that $G$ is an equivalence as well.

\end{proof}

\begin{proof}[Proof of Theorem~\ref{T:mainthm}]
  As in section~\ref{ss:finite-dimensional-dg-modules}, let
  $\cd_0(\Gamma)$ denote the localizing subcategory of $\cd(\Gamma)$
  generated by $\cd_{fd}(\Gamma)$.  Now observe that the functors $F$
  and $G$ commute with arbitrary coproducts. Since they induce
  equivalences between $\cd_{fd}(\Gamma)$ and $\cd_{fd}(\Gamma')$ and
  these subcategories are formed by compact objects, the functors $F$
  and $G$ also induce equivalences between $\cd_0(\Gamma)$ and
  $\cd_0(\Gamma')$. Now let us check that $F$ is compatible with the
  Nakayama functor $\nu$ defined in section~\ref{ss:Nakayama-functor}.
  For $P'$ in $\per(\Gamma')$ and $M$ in $\cd_0(\Gamma)$, the object
  $FP'$ is perfect in $\cd(\Gamma)$ and we have
\[
\Hom(M, \nu FP') = D\Hom(FP', M) = D\Hom(P', GM) = \Hom(GM, \nu P').
\]
Since $F$ and $G$ are quasi-inverse to each other on the subcategories
$\cd_0$, we also have
\[
\Hom(GM, \nu P') = \Hom(M, F\nu P').
\]
We obtain that when restricted to
$\per(\Gamma')$, the functors $F\nu$ and $\nu F$ are isomorphic. Since
$F\nu$ is fully faithful on $\per(\Gamma')$ and $\nu$ is fully faithful on $\per(\Gamma)$,
it follows that $F$ is fully faithful on $\per(\Gamma')$. Since $F\Gamma' =T$ generates
$\per(\Gamma)$ it follows that $F$ induces an equivalence from $\per(\Gamma)$
to $\per(\Gamma')$. Since $F$ commutes with arbitrary coproducts, $F$ itself
is an equivalence.
\end{proof}

\section{Nearly Morita equivalence for neighbouring Jacobian algebras}
\label{s:nearly-Morita-equivalence} Let $(Q,W)$ be a quiver with
potential, and $\Gamma$ the associated complete Ginzburg dg algebra.
For a vertex $j$ of $Q$, let $P_j=e_j\Gamma$ and let $S_j$ be the
corresponding simple dg $\Gamma$-module concentrated in degree~$0$.

Recall that the finite dimensional derived category
$\cd_{fd}(\Gamma)$ is generated by the $S_j$'s and all of them
belong to the perfect derived category $\per(\Gamma)$ (\confer
subsections~\ref{ss:cofibrant-resolutions} and
\ref{ss:finite-dimensional-dg-modules}). It follows that
$\cd_{fd}(\Gamma)$ is a triangulated subcategory of $\per(\Gamma)$.
We call the idempotent completion $\cc_{Q,W}$ of the
triangle quotient category
\[
\per(\Gamma)/\cd_{fd}(\Gamma)
\] 
\emph{the generalized cluster category associated with $(Q,W)$}. This
category was introduced by C.~Amiot in the case when the Jacobian
algebra $J(Q,W)=H^0\Gamma$ is finite dimensional, \confer
\cite{Amiot08a}.  Let $\pi$ denote the projection functor from
$\per(\Gamma)$ to $\cc_{Q,W}$.

\begin{remark} For a compactly generated triangulated category
$\ct$, let us denote by $\ct^c$ the subcategory of compact objects,
\ie the objects $C$ such that the functor $\Hom(C,?)$ commutes
with arbitrary coprducts, 
\confer \cite{Neeman99}.
Recall that $\cd_0(\Gamma)$ denotes the localizing subcategory
of $\cd(\Gamma)$ generated by $\cd_{fd}(\Gamma)$. We have
$(\cd \Gamma)^c=\per(\Gamma)$ and $(\cd_0 \Gamma)^c=\cd_{fd}(\Gamma)$.
Thus, by a theorem of Neeman \cite{Neeman92a}, we have an
equivalence
\[
\cc_{Q,W} \iso (\cd(\Gamma)/\cd_0(\Gamma))^c.
\]
This shows that the quotient $\cd(\Gamma)/\cd_0(\Gamma)$ is
the `unbounded version' of the cluster category.
\end{remark}

Let $\fpr(\Gamma)$ be the full subcategory of
$\cc_{Q,W}$ consisting of cones of morphisms of
$\add(\pi(\Gamma))$. If the Jacobian algebra is finite dimensional,
this subcategory equals $\cc_{Q,W}$, \confer Proposition 2.9 and
Lemma 2.10 of \cite{Amiot08a}. The following proposition generalizes
\cite[Proposition 2.9]{Amiot08a}.

\begin{proposition}\cite{Plamondon09}\label{P:fundamental-domain-of-cluster-category}
Let $\cf$ be the full subcategory of $\per(\Gamma)$ consisting of
cones of morphisms of $\add(\Gamma)$. Then the projection functor
$\pi:\per(\Gamma)\rightarrow \cc_{Q,W}$ induces a $k$-linear
equivalence between $\cf$ and $\fpr(\Gamma)$.
\end{proposition}

As immediate consequences of
Proposition~\ref{P:fundamental-domain-of-cluster-category}, we have
\[\Hom_{\cc_{Q,W}}(\Gamma,\Gamma)=\Hom_{\per(\Gamma)}(\Gamma,\Gamma)=J(Q,W),\]
and
\[\Hom_{\cc_{Q,W}}(\Gamma,\Sigma\Gamma)=\Hom_{\per(\Gamma)}(\Gamma,\Sigma\Gamma)=0.\]

\begin{proposition}
The functor $\Hom_{\cc_{Q,W}}(\Gamma,?)$ induces an equivalence
from the additive quotient category $\fpr(\Gamma)/(\Sigma\Gamma)$ to
the category $\mod J(Q,W)$ of finitely presented modules over the
Jacobian algebra $J(Q,W)$.
\end{proposition}
\begin{proof} As in \cite[Proposition 2.1
c)]{KellerReiten07}, \confer also \cite[Theorem
2.2]{BuanMarshReiten04}.
\end{proof}

We denote the functor $\Hom_{\cc_{Q,W}}(\Gamma,?)$ by $\Psi$ and
also use this symbol for the induced equivalence as in the preceding
proposition.

\begin{lemma}\cite{Plamondon09}\label{L:ses-lifts-to-triangle} Short exact sequences in $\mod
J(Q,W)$ lift to triangles in $\fpr(\Gamma)$. More precisely, given a
short exact sequence in $\mod J(Q,W)$, we can find a triangle in
$\cc_{Q,W}$ whose terms are in $\fpr(\Gamma)$ and whose image
under the functor $\Psi$ is isomorphic to the given short exact
sequence.
\end{lemma}

Assume that $Q$ does not have loops. Let $j$ be any vertex of $Q$.
In this section (in contrast to Section~\ref{S:the-equivalence}), we
denote by $T_j$ the mapping cone of the morphism of dg modules
\[P_j\stackrel{(\rho)}{\longrightarrow} \bigoplus_{\rho}P_{t(\rho)},\]
where we have a summand $P_{t(\rho)}$ for each arrow $\rho$ of $Q$
with source $j$, and the corresponding component of the morphism is
the left multiplication by $\rho$. We denote by $R_j$ the mapping
cone of the morphism of dg modules
\[\bigoplus_{\tau}P_{s(\tau)}\stackrel{(\tau)}{\longrightarrow} P_j,\]
where we have a summand $P_{s(\tau)}$ for each arrow $\tau$ of $Q$
with target $j$, and the corresponding component of the morphism is
the left multiplication by $\tau$. From $\Sigma T_j$ to $R_j$ there
is a morphism of dg modules, which we will denote by $\varphi_j$,
given in matrix form as
\[\varphi_j=\left(\begin{array}{cc} -\tau^* & -\del_{\rho\tau}W\\ t_j & \rho^*\end{array}\right):\Sigma T_j\longrightarrow R_j.\]
The mapping cone of $\varphi_j$ is isomorphic as a dg
$\Gamma$-module to the standard cofibrant resolution of $S_j$ given
in Section~\ref{ss:cofibrant-resolutions}. Thus $\Sigma T_j$ and
$R_j$ are isomorphic in $\cc_{Q,W}$. In particular, $\Sigma T_j$
belongs to $\fpr(\Gamma)$. Direct calculation shows that the
$J(Q,W)$-module $\Psi(R_j)$ and hence $\Psi(\Sigma T_j)$ are
isomorphic to the simple module $S_j$.

From now on, we fix a vertex $i$ which does not lie on a 2-cycle.
Let $(Q',W')=\tilde{\mu}_i(Q,W)$, let $\Gamma'$ be the associated
complete Ginzburg dg algebra, and let $P'_j=e_j\Gamma'$ for each
vertex $j$ and let $S'_j$ be the corresponding simple dg
$\Gamma'$-module concentrated in degree~$0$. It follows from
Theorem~\ref{T:mainthm} that we have a triangle equivalence
\[\bar F:\cc_{(Q',W')}\longrightarrow\cc_{Q,W},\]
which sends $P'_j$ to $P_j$ for $j\neq i$, and $P'_i$ to $T_i$.

\begin{proposition}\cite{Plamondon09}\label{P:stability-mutation}
The equivalence $\bar F:\cc_{(Q',W')}\rightarrow \cc_{Q,W}$
restricts to an equivalence from $\fpr(\Gamma')$ to $\fpr(\Gamma)$.
\end{proposition}

\begin{corollary}\label{C:nearly-morita-equivalence} There is a canonical equivalence of
categories of finitely presented modules
\[\mod J(Q',W')/(S'_i)\longrightarrow \mod J(Q,W)/(S_i),\]
which restricts to an equivalence between the subcategories of
finite dimensional modules.
\end{corollary}

\begin{remark} The restriction of the equivalence in Corollary~\ref{C:nearly-morita-equivalence} to
the subcategories of finite dimensional modules is a nearly Morita
equivalence in the sense of Ringel \cite{Ringel07}. It was first
shown in \cite[Theorem 7.1]{BuanIyamaReitenSmith08} that  the two
Jacobian algebras $J(Q,W)$ and $J(Q',W')$ are nearly Morita
equivalent in the sense of Ringel. Another proof was given in
\cite[Proposition 6.1]{DerksenWeymanZelevinsky09}.
\end{remark}

\begin{proof}
For each vertex $j$ of $Q'$, we define dg $\Gamma'$-modules $T'_j$
and $R'_j$ analogously to the dg $\Gamma$-modules $T_j$ and $R_j$.
Explicitly, $R'_i$ is the mapping cone of the morphism of dg modules
\[\bigoplus_{\alpha}P'_{t(\alpha)}\stackrel{(\alpha^{\star})}{\longrightarrow} P'_i,\]
where $\alpha$ runs over all arrows of $Q$ with source $i$. The dg
functor underlying $F$ takes this morphism of dg modules to the
morphism
\[\bigoplus_{\alpha}P_{t(\alpha)}\stackrel{(f_{\alpha^{\star}})}{\longrightarrow} T_i,\]
where $f_{\alpha^{\star}}$ was defined in Section~\ref{ss:bimodule}.
In fact, it is exactly the inclusion of the summand $P_{t(\alpha)}$
corresponding to the arrow $\alpha$ into $T_i$. We immediately
deduce that the mapping cone of this latter morphism is
quasi-isomorphic to $\Sigma P_i$. So the functor $F$ takes $R'_i$ to
$\Sigma P_i$. Thus we have isomorphisms in $\cc_{Q,W}$
\[\Sigma\bar F(T'_i)\cong \bar F(R'_i)\cong\Sigma P_i.\]

 By
Proposition~\ref{P:stability-mutation}, the equivalence $\bar F$
induces an equivalence
\[\fpr(\Gamma')/(\Sigma\Gamma'\oplus\Sigma T'_i)\longrightarrow\fpr(\Gamma)/(\Sigma\Gamma\oplus\Sigma T_i).\]
Thus, we have the following square
\[\xymatrix{ \fpr(\Gamma')/(\Sigma\Gamma'\oplus\Sigma T'_i)\ar[r]^{\simeq}\ar[d]^{\simeq} & \mod(J(Q',W'))/(S'_i)\ar @{-->}[d]^\Phi\\
\fpr(\Gamma)/(\Sigma\Gamma\oplus\Sigma T_i)\ar[r]^{\simeq} &
\mod(J(Q,W))/(S_i)}
\]
The induced functor $\Phi$ from $\mod(J(Q',W'))/(S'_i)$ to
$\mod(J(Q,W))/(S_i)$, represented by the dashed arrow, is an
equivalence as well. Explicitly, for a finitely presented module $M$
over $J(Q,W')$, the module $\Phi(M)$ is defined as $\Psi\circ\bar F
(X)$, where $X$ is an object of $\fpr(\Gamma')$ such that
$\Psi'(X)\cong M$.

It remains to prove that the functor $\Phi$ and its quasi-inverse
$\Phi^{-1}$ take finite dimensional modules to finite dimensional
modules. We will prove the assertion for $\Phi$ by induction on the
dimension of the module: it is similar for $\Phi^{-1}$.

Let $j$ be a vertex of $Q$. We have seen that the simple
$J(Q',W')$-module $S'_j$ is the image of $\Sigma T'_j$ under the
functor $\Psi'$. First of all, we have
\[\Phi(S'_i)=\Psi\circ\bar F (\Sigma T'_i)=\Psi(\Sigma P_i)=0.\]
In the following we assume that $j$ is different from $i$. We have a
triangle in $\cc_{(Q',W')}$:
\[\xymatrix{\bigoplus_{\rho}P'_{s(\rho)}\ar[r]^(.6){(\rho)}&P'_j\ar[r]&
\Sigma T'_j\ar[r]&\Sigma\bigoplus_{\rho}P'_{s(\rho)},}\] where
$\rho$ ranges over all arrows of $Q'$ with target $j$.   We apply
the functor $\bar F$ and obtain a triangle in $\cc_{Q,W}$:
\[\xymatrix{
\bigoplus_{\rho}P_{s(\rho)}\oplus\bigoplus_{\alpha,b}P_{s(b)}\oplus\bigoplus_{\beta}T_i\ar[rr]^(.75){(f_{\rho},f_{[\alpha
b]},f_{\beta^{\star}})} &&P_j\ar[r] & \Sigma F(T'_j)}\]
\[\xymatrix{\ar[rr]&& \bigoplus_{\rho} \Sigma
P_{s(\rho)}\oplus\bigoplus_{\alpha,b}\Sigma
P_{s(b)}\oplus\bigoplus_{\beta}\Sigma T_i,}\] where $\rho$ ranges
over arrows of $Q$ with target $j$ and with source different from
$i$, $\alpha$ ranges over all arrows of $Q$ with source $i$ and
target $j$, $b$ ranges all arrows of $Q$ with target $i$ and $\beta$
ranges over all arrows of $Q$ with source $j$ and target $i$. Note
that the first two terms from the left are in the fundamental domain
$\cf$. Thus, applying the functor $\Psi$, we obtain an exact
sequence
\[\xymatrix{H^0(\bigoplus_{\rho}P_{s(\rho)}\oplus\bigoplus_{\alpha,b}P_{s(b)}\oplus\bigoplus_{\beta}T_i)\ar[rr]^(.75){H^0(f_{\rho},f_{[\alpha b]},f_{\beta^{\star}})}&& H^0 P_j\ar[r] & \Phi(S'_j)
\ar[r]& \bigoplus_{\beta}S_i.}\] The image of the first map is the
$J(Q,W)$-submodule of $H^0P_j$ generated by $\rho$, $\alpha b$, and
$\del_{a\beta}W$, where $\rho$, $\alpha$, $\beta$, $a$ and $b$ are
arrows of $Q$ such that $s(\rho)\neq i$,
$t(\rho)=t(\alpha)=s(\beta)=j$, $s(\alpha)=t(b)=s(a)=t(\beta)=i$. It
has finite codimension in $H^0P_j$. Therefore, the module
$\Phi(S'_j)$ is finite dimensional.

Now let $M$ be a finite dimensional $J(Q',W')$-module. Suppose
\[\xymatrix{0\ar[r]&L\ar[r]&M\ar[r]&N\ar[r]&0}\]
is a short exact sequence in $\mod J(Q',W')$ such that $L$ and $N$
are nontrivial. Following Lemma~\ref{L:ses-lifts-to-triangle}, we
have a triangle in $\fpr(\Gamma')$:
\[\xymatrix{X\ar[r]& Y\ar[r] & Z\ar[r] & \Sigma X,}\]
whose image under the functor $\Psi'$ is isomorphic to the above
short exact sequence. We apply the functor $\bar F$ to this triangle
and obtain a triangle in $\fpr(\Gamma)$:
\[\xymatrix{\bar F X\ar[r]& \bar F Y\ar[r] & \bar F Z\ar[r] & \Sigma \bar F X.}\]
Applying $\Psi$ to this triangle, we obtain an exact sequence in
$\mod J(Q,W)$:
\[\xymatrix{\Phi(L)\ar[r] & \Phi(M)\ar[r] & \Phi(N).}\]
By induction hypothesis, both $\Phi(L)$ and $\Phi(N)$ are finite
dimensional, and hence so is $\Phi(M)$. This completes the proof.
\end{proof}

\section{Tilting between Jacobian algebras}
\label{s:tilting-Jacobian-algebras}

\subsection{The canonical $t$-structure} \label{ss:canonical-t-structure}
Let $A$ be a dg $k$-algebra such that the homology $H^p (A)$
vanishes for all $p>0$. Let $\cd_{\leq 0}$ be the full subcategory
of the derived category $\cd(A)$ whose objects are the dg modules
$M$ such that the homology $H^p(M)$ vanishes for all $p>0$. The
following lemma already appears in section~2.1 of \cite{Amiot08a}.
For the convenience of the reader, we include a detailed proof
(slightly different from the one in \cite{Amiot08a}).

\begin{lemma} \label{lemma:canonical-t-structure}
\begin{itemize}
\item[a)] The subcategory $\cd_{\leq 0}$ is a left aisle
\cite{KellerVossieck88} in the derived category $\cd(A)$.
\item[b)] The functor $M \mapsto H^0(M)$ induces an equivalence from the heart
of the corresponding $t$-structure to the category $\Mod H^0(A)$
of all right $H^0(A)$-modules.
\end{itemize}
\end{lemma}

\begin{proof} a) Clearly, the subcategory $\cd_{\leq 0}$ is
stable under the suspension functor $\Sigma$ and under extensions.
We have to show that for each object $M$ of $\cd(A)$, there
is a triangle
\[
M' \to M \to M'' \to \Sigma M'
\]
such that $M'$ belongs to $\cd_{\leq 0}$ and $M''$ is
right orthogonal to $\cd_{\leq 0}$.
The map of complexes $\tau_{\leq 0} A \to A$ is a
quasi-isomorphism of dg algebras. Thus we may assume that the
components $A^p$ vanish for all $p>0$. If $M$ is a dg module,
the subcomplex $\tau_{\leq 0} M$ of $M$ is then a dg submodule
and thus the sequence
\[
0 \to \tau_{\leq 0} M \to M \to \tau_{> 0} M \to 0
\]
is an exact sequence of dg modules. Of course, the two truncation
functors preserve quasi-isomorphisms and thus induce functors
from $\cd(A)$ to itself so that the above sequence yields
a triangle functorial in the object $M$ of $\cd(A)$. In this
triangle, the object $\tau_{\leq 0} M$ lies in of course in
$\cd_{\leq 0}$. Let us show that $\tau_{>0} M$ lies in the
right orthogonal subcategory of $\cd_{\leq 0}$. Indeed, if
$L$ lies in $\cd_{\leq 0}$ and we have a morphism $f: L \to \tau_{>0} M$,
then, by the functoriality of the triangle, $f$ factors through
$\tau_{>0} L$, which vanishes.

b) As in the proof of a), we may assume that the components $A^p$
vanish for $p>0$. We then have the morphism of dg algebras $A \to H^0(A)$.
The restriction along this morphism yields a functor from
$\Mod H^0(A)$ to the heart $\ch$ of the $t$-structure. Using
the truncation functors defined in a), we see that this functor
is essentially surjective. Let us show that it is fully faithful.
Let $L$ and $M$ in be in $\Mod H^0(A)$.
We compute morphisms between their images in
$\cd(A)$. We have
\[
(\cd(A))(L,M) = \colim \ch(A)(L, M')
\]
where $M'$ ranges through the category of quasi-isomorphisms with
source $M$ in $\ch(A)$. Now for each object $M \to M'$ in this
category, we have the object $M \to M' \to \tau_{\geq 0} M'$. This
shows that the objects $M'$ with $M'^n=0$ for $n<0$ form a cofinal
subcategory.  We restrict the colimit to this cofinal subcategory.
We find
\[
\colim \ch(A)(L, M') = \colim \ch(A)(L, \tau_{\leq 0}
M') \iso \colim \ch(A)(L, H^0(M')) = \ch(A)(L,M).
\]
Thus the functor $\Mod H^0(A) \to \cd(A)$ is fully faithful.
\end{proof}

\subsection{Comparison of t-structures}\label{ss:comparison-of-t-structures}
Assume that $(Q,W)$ and $(Q',W')$ are two quivers with potential
related by the mutation at a vertex $i$. Let $\Gamma$ and $\Gamma'$
be the corresponding complete Ginzburg algebras and
\[
F: \cd(\Gamma') \to \cd(\Gamma)
\]
the associated triangle equivalence taking $P'_j=e_j \Gamma'$
to $P_j$ for $j\neq i$ and $P_i'$ to the cone over the morphism
\[
P_i \to \bigoplus_{\alpha: s(\alpha)= i} P_{t(\alpha)}
\]
whose components are the left multiplications by the corresponding arrows
$\alpha$. The image under $F$ of the canonical $t$-structure
on $\cd(\Gamma')$ (\confer section~\ref{ss:canonical-t-structure})
is a new $t$-structure on $\cd(\Gamma)$. Let us
denote its left aisle by $\cd_{\leq 0}'$ and its right
aisle by $\cd_{>0}'$. Let us denote
the left aisle of the canonical $t$-structure on $\cd(\Gamma)$
by $\cd_{\leq 0}$ and let $\ca$ be its heart.
By lemma~\ref{lemma:canonical-t-structure} b),
the category $\ca$ is equivalent to the category $\Mod(J(Q,W))$
of all right modules over the Jacobian algebra of $(Q,W)$.

The following lemma shows how the new $t$-structure is obtained from
the old one and that the two hearts are `piecewise equivalent'.
However, in general, the new heart is not tilted from the old one in
the sense of \cite{HappelReitenSmaloe96}, \confer also
\cite{Bridgeland05}, because these hearts are not faithful in
general (\ie the higher extension groups computed in the hearts are
different from the higher extension groups computed in the ambient
triangulated categories). They are, however, faithful if the
homology of $\Gamma$ is concentrated in degree~$0$ and then the
homology of $\Gamma'$ is also concentrated in degree~$0$, as we show
below in theorem~\ref{thm:CY-Jacobi-algebras}.

\begin{lemma} \label{lemma:comparison-of-t-structures}
\begin{itemize}
\item[a)] We have $\Sigma \cd_{\leq 0} \subset
\cd'_{\leq 0} \subset \cd_{\leq 0}$.
\item[b)] Let $\cf\subset\ca$ be the subcategory of modules $M$ supported at
$i$ and $\ct\subset\ca$ the left orthogonal subcategory of $\cf$.
Then an object $X$ of $\cd(\Gamma)$ belongs to $\cd'_{\leq 0}$
(respectively $\cd'_{>0}$) iff
$H^n(X)$ vanishes in all degrees $n>0$ and $H^0(X)$ lies in $\ct$
(respectively $H^n(X)$ vanishes in all degrees $n<0$ and $H^0(X)$
lies in $\cf$).
\item[c)] The pair $(\cf,\ct)$ is a torsion pair, \ie  $\cf$ is the right
orthogonal of $\ct$ and $\ct$ the left orthogonal of $\cf$.
\end{itemize}
\end{lemma}

\begin{proof} An object $X$ of $\cd(\Gamma)$ belongs to
  $\cd_{\leq 0}'$ iff it satisfies $\Hom(FP'_j,\Sigma^n X)=0$ for all
  $n>0$ and all $1\leq j\leq n$.  Using this characterization, parts
  a) and b) are easy to check. Also part c) is easy to check using
  that $\cd'_{\leq 0}$ and $\cd'_{>0}$ are the two aisles of a
  $t$-structure.
\end{proof}

\begin{corollary} \label{cor:new-t-structure}
  Let $\ca'$ be the heart of the
  $t$-structure $(\cd'_{\leq 0},\cd'_{\geq 0})$. Then an object $X$ of
  $\cd(\Gamma)$ belongs to $\ca'$ if and only if $H^n(X)=0$ for
  $n\neq 0,-1$, $H^0(X)$ belongs to $\ct$ and $H^{-1}(X)$ belongs to
  $\cf$.
\end{corollary}

\section{Stability under mutation of Ginzburg algebras\\ concentrated
in degree~$0$} \label{s:stability-under-mutation}

\subsection{The statement}
Let $k$ be a field. Let $Q$ be a non-empty finite quiver and $W$ a
potential in $\hat{kQ}$. The quiver $Q$ may contain loops and
$2$-cycles but we assume that $W$ is reduced (\ie no paths of
length $\leq 1$ occur in the relations deduced from $W$). Let
$\Gamma$ be the complete Ginzburg algebra associated with $(Q,W)$
and put $A=H^0(\Gamma)$.

For each vertex $i$ of $Q$, we denote by $S_i$ the associated simple
module and by $P_i = e_i A$ its projective cover.
Fix a vertex $i$ of $Q$ and consider the complex $T'$ which is the
sum of the $P_j$, $j\neq i$, concentrated in degree~$0$ and of the
complex
\[
\xymatrix{0 \ar[r] & P_i \ar[r]^f & B \ar[r] & 0}
\]
where $P_i$ is in degree $-1$, $B$ is the sum over the
arrows $\alpha$ starting in $i$ of the $P_{t(\alpha)}$
and the components of $f$ are the left multiplications
by the corresponding arrows.

\begin{theorem} \label{thm:CY-Jacobi-algebras}
Suppose that the complete Ginzburg algebra $\Gamma=\hat\Gamma(Q,W)$ has its
homology concentrated in degree~$0$. If the simple module $S_i$
corresponding to the vertex $i$ is spherical (\ie there are no
loops at $i$ in $Q$), then $T'$ is a tilting object in the perfect
derived category $\per(A)$. Thus, under the assumptions (c1) (c2)
and (c3), the complete Ginzburg algebra $\Gamma'$ associated with
the mutated quiver with potential $\mu_i(Q,W)$ still has its
homology concentrated in degree~$0$ and the Jacobian algebras $A$ and $A'$
associated with $(Q,W)$ and $\mu_i(Q,W)$ are derived equivalent.
\end{theorem}

Using the theorem and corollary~\ref{cor:new-t-structure}
we obtain the following corollary.

\begin{corollary} Under the hypothesis of the theorem,
the category $\Mod J(\mu_i(Q,W))$ is obtained by tilting, in the
sense of \cite{HappelReitenSmaloe96} \cite{Bridgeland05}, from the
category $\Mod J(Q,W)$ at the simple module $S_i$.
\end{corollary}

Let us prove the theorem.

\begin{proof} We would like to use proposition~\ref{proposition:tilting}
  below and have to check conditions 1), 2) and 3). Condition 1) holds
  since $B$ belongs to $\cb$, the additive closure of the $P_j$, $j\neq i$. Condition 2) holds since $f: P_i \to B$
  is a left $\cb$-approximation.  Finally, in order to show condition
3), it suffices to show that $f$ is injective. This can be deduced
from Ginzburg's results \cite{Ginzburg06} but we can also show it as
follows: Since the homology of $\Gamma$ is concentrated in
degree~$0$, the functor $?\lten_\Gamma A$ is an equivalence from
$\per(\Gamma)$ to $\per(A)$ whose inverse is given by the
restriction along the projection morphism $\Gamma \to H^0\Gamma =
A$. If we apply the equivalence $?\lten_\Gamma A$ to the cofibrant
resolution $\bp S_i$ constructed in
section~\ref{ss:cofibrant-resolutions}, we obtain the complex of
projective $A$-modules
\[
\xymatrix{0 \ar[r] & P_i \ar[r]^{f} & B \ar[r] &  B' \ar[r] &
P_i \ar[r] & 0} \ko
\]
as we see by using the explicit description of $\bp S_i$ given at
the beginning of the proof of lemma~\ref{L:imageofsimples}.
The image of this complex under the restriction along
$\Gamma \to A$ is again quasi-isomorphic to $S_i$. So the
complex itself is quasi-isomorphic to $S_i$. In particular,
the map $f$ is injective.
\end{proof}

\subsection{Tilts of tilting objects in triangulated categories}
The following lemma is well-known to the experts in tilting theory
although it is not easy to point to a specific reference. One could
cite Auslander-Platzeck-Reiten \cite{AuslanderPlatzeckReiten79},
Riedtmann-Schofield \cite{RiedtmannSchofield91}, Happel-Unger
\cite{HappelUnger05}, Buan-Marsh-Reineke-Reiten-Todorov
\cite{BuanMarshReinekeReitenTodorov06}, Geiss-Leclerc-Schr\"oer
\cite{GeissLeclercSchroeer06} and many others.

Let $\ct$ be a triangulated category with suspension functor
$\Sigma$. Let $T$ be a {\em tilting object in $\ct$}, \ie
\begin{itemize}
\item[a)] $\ct$ coincides with the closure of $T$ under taking
suspensions, desuspensions, extensions and direct factors and
\item[b)] we have $\ct(T,\Sigma^n T)=0$ for all integers $n\neq 0$.
\end{itemize}
Assume that we are given a decomposition
\[
T=T_0 \oplus T_1.
\]
Let $\add(T_1)$ denote the closure of $T_1$ under
taking finite direct sums and direct summands.
Assume that there exists a map $f: T_0 \to B$ such that
\begin{itemize}
\item[1)] $B$ belongs to $\add(T_1)$;
\item[2)] the map $f^*: \ct(B,T_1) \to \ct(T_0, T_1)$ is surjective and
\item[3)] the map $f_*: \ct(T_1, T_0) \to \ct(T_1, B)$ is injective.
\end{itemize}
Choose a triangle
\begin{equation} \label{eq:triangle}
\xymatrix{ T_0 \ar[r]^f & B \ar[r]^g & T_0^* \ar[r]^-h & \Sigma T_0}.
\end{equation}
\begin{proposition} \label{proposition:tilting}
The object $T'=T_0^*\oplus T_1$ is a tilting object in $\ct$.
\end{proposition}

\begin{proof} It is clear from the triangle (\ref{eq:triangle}) that $T'$ still
generates $\ct$ in the sense of condition a). It remains to be checked that there are
no morphisms between non trivial shifts of the two given summands of $T'$. We
distinguish several cases:

{\em Case 1: We have $\ct(T_0^*, \Sigma T_0^*)=0$.} Consider the triangles
\begin{align*}
& \xymatrix{
B \ar[r]^g & T_0^*  \ar[r]^h & \Sigma T_0 \ar[r]^{-\Sigma f} & \Sigma B}\ko \\
& \xymatrix{\Sigma B \ar[r]_{\Sigma g} & \Sigma T_0^* \ar[r]_{\Sigma h} & \Sigma^2 T_0 \ar[r] & \Sigma^2 B.
}
\end{align*}
Using condition b) and the long exact sequences associated with these
triangles we obtain the following diagram whose first two rows and
last column are exact (we write $(X,Y)$ instead of $\ct(X,Y)$)
\[
\xymatrix{
(\Sigma B, \Sigma T_0) \ar[r] \ar[d] & (\Sigma B, \Sigma B) \ar[r] \ar[d] & (\Sigma B, \Sigma T_0^*)\ar[d] \ar[r] & 0 \\
(\Sigma T_0, \Sigma T_0) \ar[r] & (\Sigma T_0, \Sigma B) \ar[r] & (\Sigma T_0, \Sigma T_0^*) \ar[r] \ar[d] & 0 \\
 & & (T_0^*, \Sigma T_0^*) \ar[d] \\
 & &   0.}
 \]
This implies that the space of morphisms $(T_0^*, \Sigma T_0^*)$ is
isomorphic to the space of maps up to homotopy from the complex $T_0
\to B$, with $T_0$ in degree~$0$, to its shift by one degree to the
left. Clearly conditions 1) and 2) suffice for this space to vanish.

{\em Case 2: We have $\ct(T_0^*, \Sigma^{-1}T_0^*)=0$.} Consider the triangles
\begin{align*}
& \xymatrix{
B \ar[r]^g & T_0^* \ar[r]^h & \Sigma T_0 \ar[r]^{-\Sigma f} & \Sigma B } \\
& \xymatrix{\Sigma^{-1} B \ar[r]_{\Sigma^{-1} g} & \Sigma^{-1} T_0^* \ar[r]_{\Sigma^{-1} h} &  T_0 \ar[r]_f & B.}
\end{align*}
By using condition b) and the long exact sequences associated with these triangles
we obtain the following commutative diagram whose first column and two last rows
are exact
\[
\xymatrix{
 & 0 \ar[d] & & \\
 & (T_0^*, \Sigma^{-1} T_0^*) \ar[d] & & \\
0 \ar[r] & (B, \Sigma^{-1} T_0^*) \ar[d] \ar[r] & (B,T_0) \ar[d] \ar[r] & (B,B) \ar[d] \\
0 \ar[r] & (T_0, \Sigma^{-1} T_0^*) \ar[r] & (T_0, T_0) \ar[r] & (T_0, B).
}
\]
This shows that the space of morphisms $(T_0^*, \Sigma^{-1} T_0^*)$
is isomorphic to the space of maps up to homotopy from the complex
$T_0 \to B$, with $T_0$ in degree~$0$, to its shift by one degree to
the right. Clearly conditions~2) and 3) suffice to imply the
vanishing of this space.

{\em Case 3: We have $\ct(T_0^*, \Sigma^n T_0^*)=0$ for all integers
$n$ different from $-1$, $0$ and $1$.} This follows easily by
considering the long exact sequences associated with the triangle
(\ref{eq:triangle}).

{\em Case 4: We have $\ct(T_0^*, \Sigma^n T_1)=0$ for all integers
$n$ different from $0$ and $1$.} Again this is easy.

{\em Case 5: We have $\ct(T_0^*, \Sigma T_1)=0$.} This follows if we
apply $\ct(?,\Sigma T_1)$ to the triangle
\[
\xymatrix{ B \ar[r]^g & T_0^* \ar[r]^h & \Sigma T_0 \ar[r]^{-\Sigma
f} & \Sigma B}
\]
and use condition 2).

{\em Case 6: We have $\ct(T_1, \Sigma^n T_0^*)=0$ for $n\neq 0$.} This follows
if we apply $\ct(T_1, ?)$ to the triangle
\[
\xymatrix{ \Sigma^n B \ar[r]^{\Sigma^n f} & \Sigma^n T_0^* \ar[r]^{\Sigma^n h} &
\Sigma^{n+1} T_0 \ar[r] & \Sigma^{n+1} B}
\]
and use condition 3) in the case where $n=-1$.

{\em Case 7: We have $\ct(T_1, \Sigma^n T_1)=0$.} This is clear
since $T$ is tilting.

\end{proof}

\section{Appendix: Pseudocompact dg algebras and derived categories\\
by Bernhard Keller} \label{s:pseudocompact}

\subsection{Motivation} As pointed out by D.~Simson \cite{Simson07a},
the Jacobian algebra of a quiver with potential is a topological
algebra in a natural way and it is natural and useful to take this
structure into account and to consider (suitable) topological
modules (or dually, to consider coalgebras and comodules, \confer
for example \cite{Simson07}). It is then necessary to consider
derived categories of topological modules as well. This allows one,
for example, to extend the results of Amiot \cite{Amiot08a} from
potentials belonging to $kQ$ to potentials belonging to the
completion $\hat{kQ}$. We briefly recall some classical facts on the
relevant class of topological algebras and modules: the
pseudocompact ones. We show their usefulness in reconstructing
Jacobian algebras from their categories of finite-dimensional
modules. Then we adapt the notion of pseudocompact module to the dg
setting and show that mutation also yields equivalences between
pseudocompact derived categories.

\subsection{Reminder on pseudocompact algebras and modules}
\label{ss:reminder-pseudocompact} Let $k$ be a field and $R$ a
finite-dimensional separable $k$-algebra (\ie $R$ is projective as a
bimodule over itself). By an {\em $R$-algebra}, we mean an algebra
in the monoidal category of $R$-bimodules. Following chapter~IV,
section~3 of \cite{Gabriel62}, we call an $R$-algebra $A$ {\em
pseudocompact} if it is endowed with a linear topology for which it
is complete and separated and admits a basis of neighborhoods of
zero formed by left ideals $I$ such that $A/I$ is finite-dimensional
over $k$. Let us fix a pseudocompact $R$-algebra $A$. A right
$A$-module $M$ is {\em pseudocompact} if it is endowed with a linear
topology admitting a basis of neighborhoods of zero formed by
submodules $N$ of $M$ such that $M/N$ is finite-dimensional over
$k$. A {\em morphism} between pseudocompact modules is a continuous
$A$-linear map. We denote by $\Pcm(A)$ the {\em category of
pseudocompact right $A$-modules} thus defined. For example, if $A$
equals $R$, then we have the equivalence
\[
M \mapsto \Hom_R(M,R).
\]
from the opposite of the category $\Mod R$ to $\Pcm(A)$. In
particular, since $\Mod R$ is semisimple, so is $\Pcm(R)$.

Recall that a {\em Grothendieck category} is an abelian category
where all colimits exist, where the passage to a filtered colimit is
an exact functor and which admits a generator; a {\em locally finite
category} is a Grothendieck category where each object is the
filtered colimit of its subobjects of finite length.

\begin{theorem}\cite{Gabriel62} \label{thm:pseudocompact-modules}
\begin{itemize}
\item[a)] The category $\Pcm(A)$ is an abelian category and the
forgetful functor $\Pcm(A) \to \Mod R$ is exact.
\item[b)] The opposite category $\Pcm(A)\op$ is a locally finite category.
\end{itemize}
\end{theorem}

The main ingredient of the proof of the theorem is the following
well-known `Mittag-Leffler lemma'.
\begin{lemma} \label{lemma:Mittag-Leffler} If
\[
\xymatrix{0 \ar[r] & L_i \ar[r] & M_i \ar[r] & N_i \ar[r] & 0 } \ko
i\in I\ko
\]
is a filtered inverse system of exact sequences of vector spaces and
all the $L_i$ are finite-dimensional, then the inverse limit of the
sequence is still exact.
\end{lemma}

Using classical facts on locally finite Grothendieck categories we
deduce the following corollary from the theorem.

\begin{corollary} \label{cor:good-properties-of-pcmA}
\begin{itemize}
\item[a)] The dual of the Krull-Schmidt theorem holds in $\Pcm(A)$ (decompositions into
possibly infinite {\em products} of objects with local endomorphism
rings are unique up to permutation and isomorphism).
\item[b)]  Each object of the category $\Pcm(A)$ admits a projective cover and a
minimal projective resolution.
\item[c)] The category $\Pcm(A)$ is determined, up to equivalence, by its full subcategory
of finite-length objects $\mod A$. More precisely, we have the
equivalence
\[
\Pcm(A) \to \Lex(\mod A, \Mod k) \ko M \mapsto \Hom_A(M,?)|\mod A
\ko
\]
where $\Lex(\mod A, \Mod k)$ is the category of left exact functors
from $\mod A$ to the category $\Mod k$ of vector spaces.
\item[d)] The category $\Pcm(A)$ determines the pseudocompact algebra
$A$ up to isomorphism. Namely, the algebra
$A$ is the endomorphism algebra of any minimal projective generator
of $\Pcm(A)$.
\end{itemize}
\end{corollary}
\begin{proof} a) Indeed, the Krull-Schmidt theorem holds
in Grothendieck abelian categories by Theorem~1, section~6,
Chapter~I of \cite{Gabriel62}.

b) Indeed,  injective hulls exist in an arbitrary Grothendieck
category by theorem~2, section 6, Chapter~III of \cite{Gabriel62}.

c) This equivalence is a special case of the general description of
the smallest locally finite (or, more generally, locally noetherian)
abelian category containing a given finite length category, \confer
Theorem~1, section~4, Chapter~II of \cite{Gabriel62}.

d) This follows from the description of finite length categories via
pseudocompact algebras in section~4, Chapter~IV of \cite{Gabriel62}.
\end{proof}

\subsection{Pseudocompact modules over Jacobian algebras}
\label{ss:pseudocompact-Jacobian} Let $k$ be a field of
characteristic $0$. Let $Q$ be a non-empty finite quiver and $W$ a
potential in $\hat{kQ}$. The quiver $Q$ may contain loops and
$2$-cycles but we assume that $W$ is reduced, \ie no cycles of
length $\leq 2$ appear in $W$ (and thus no paths of length $\leq 1$
occur in the relations deduced from $W$).  Let $\Gamma$ be the
complete Ginzburg algebra associated with $(Q,W)$ and put
$A=H^0(\Gamma)$.

\begin{lemma} The algebra $A$ is pseudocompact in the sense of
section~\ref{ss:reminder-pseudocompact}, \ie endowed with its
natural topology, it is complete and separated and there is a basis
of neighborhoods of the origin formed by right ideals $I$ such that
$A/I$ is finite-dimensional over $k$.
\end{lemma}

\begin{proof} This follows from the fact that $\Gamma$ is
pseudocompact as a dg algebra and
lemma~\ref{lemma:pseudocompact-homology} below.
\end{proof}

We consider the category $\Pcm(A)$ of {\em pseudocompact} modules
over $A$, \ie topological right modules which are complete and
separated and have a basis of neighborhoods of the origin formed by
submodules of finite codimension over $k$, \confer
section~\ref{ss:reminder-pseudocompact}. As recalled there, the
category $\Pcm(A)$ is abelian and its opposite category $\Pcm(A)\op$
is a locally finite Grothendieck category and therefore has all the
good properties enumerated in
Corollary~\ref{cor:good-properties-of-pcmA}. If we combine parts c)
and d) of that Corollary, we obtain the following corollary.

\begin{corollary}  \label{cor:fin-dim-modules-determine-Jacobi-algebra}
The category $\mod A$ of finite-dimensional
$A$-modules determines the algebra $A$ up to isomorphism.
\end{corollary}

This yields a positive answer to question~12.1 of
\cite{DerksenWeymanZelevinsky08},
where it was asked whether the category of finite-dimensional modules
over the Jacobian algebra of a quiver with potential determined
the Jacobian algebra.

\subsection{The pseudocompact derived category of a pseudocompact
algebra} Let $A$ be a pseudocompact $R$-algebra as in
section~\ref{ss:reminder-pseudocompact}. We define the {\em
pseudocompact derived category of $A$} by $
\cd_{pc}(A)=\cd(\Pcm(A)). $ Notice that the opposite category
$\cd_{pc}(A)\op$ is isomorphic to $\cd(\Pcm(A)\op)$ and that
$\Pcm(A)\op$ is a locally finite Grothendieck category. As shown in
\cite{Franke97}, \confer also \cite{KashiwaraSchapira05}, the
category of complexes over a Grothendieck category admits a
structure of Quillen model category whose weak equivalences are the
quasi-isomorphisms and whose cofibrations are the monomorphisms.
This applies in particular to the category $\Pcm(A)\op$. Notice that
the free $A$-module $E=\,_A A$ is an injective cogenerator of this
category and that the algebra $A$ is canonically isomorphic to the
endomorphism algebra of $E$ both in $\Pcm(A)$ and in $\cd_{pc}(A)$.

\begin{lemma} If a locally finite category $\ca$ is of finite homological
dimension, then the derived category $\cd\ca$ is compactly generated
by the simple objects of $\ca$. In particular, if the category
$\ca=\Pcm(A)$ is of finite homological dimension, then the
opposite $\cd_{pc}(A)\op$ of the
pseudocompact derived category  is compactly generated
by the simple $A$-modules.
\end{lemma}

\begin{proof} If $\ca$ is of finite homological dimension, then a complex
over $\ca$ is fibrant-cofibrant for the above model structure iff it
has injective components. Moreover, since $\ca$ is in particular
locally noetherian, arbitrary coproducts of injective objects are
injective. Thus, to check that a simple object $S$ is compact, it
suffices to check that any morphism from $S$ to a coproduct of a
family of complexes with injective components factors through a
finite sub-coproduct. This is clear. Let $\ct$ be the closure in
$\cd\ca$ of the simple objects under suspensions, desuspensions,
extensions and arbitrary coproducts. Since $\ct$ is a compactly
generated subcategory of $\cd\ca$, to conclude that $\ct$ coincides
with $\cd\ca$, it suffices to check that the right orthogonal of
$\ct$ in $\cd\ca$ vanishes. Clearly $\ct$ contains all objects of
finite length. These form a generating family for $\ca$. Now it is
easy to check that a complex $I$ of injectives is right orthogonal
to $\ct$ in the category of complexes modulo homotopy iff $I$ is
acyclic. Thus the right orthogonal subcategory of $\ct$ in $\cd\ca$
vanishes and $\ct$ equals $\cd\ca$.
\end{proof}

\subsection{Extension to the dg setting}
\label{ss:extension-to-dg-setting} Let $k$ and $R$ be as in
section~\ref{ss:reminder-pseudocompact}. Let $A$ be a differential
algebra in the category of graded $R$-bimodules. We assume that $A$
is {\em bilaterally pseudocompact}, \ie it is endowed with a
complete separated topology admitting a basis of neighborhoods of
$0$ formed by bilateral differential graded ideals $I$ of finite
total codimension in $A$.

For example, if $Q$ is a finite graded quiver, we can take $R$ to be
the product of copies of $k$ indexed by the vertices of $Q$ and $A$
to be the completed path algebra, \ie for each integer $n$, the
component $A^n$ is the product of the spaces $kc$, where $c$ ranges
over the paths in $Q$ of total degree $n$. We endow $A$ with a
continuous differential sending each arrow to a possibly infinite
linear combination of paths of length $\geq 2$. For each $n$, we
define $I_n$ to be the ideal generated by the paths of lengh $\geq
n$ and we define the topology on $A$ to have the $I_n$ as a basis of
neighborhoods of $0$. Then $A$ is bilaterally pseudocompact.

A right dg $A$-module is {\em pseudocompact} if it is endowed with a
topology for which it is complete and separated (in the category of
graded $A$-modules) and which admits a basis of neighborhoods of $0$
formed by dg submodules of finite total codimension. Clearly $A$ is
a pseudocompact dg module over itself.

\begin{lemma} \label{lemma:pseudocompact-homology}
\begin{itemize}
\item[a)]
The homology $H^*(A)$ is a bilaterally pseudocompact graded
$R$-algebra. In particular, $H^0(A)$ is a pseudocompact $R$-algebra.
\item[b)]
For each pseudocompact dg module $M$, the homology $H^*(M)$ is a
pseudocompact graded module over $H^*(A)$.
\end{itemize}
\end{lemma}

\begin{proof} a) By the Mittag-Leffler lemma~\ref{lemma:Mittag-Leffler},
the homology $H^*(A)$ identifies with the limit of the $H^*(A/I)$, where
$I$ runs through a basis of neighbourhoods of zero formed by bilateral
dg ideals of finite codimension. This inverse limit is also
the inverse limit of the images of the maps $H^*(A) \to H^*(A/I)$.
Now clearly, the kernels of these maps are of finite codimension
and form a basis of neighborhoods of zero for a complete separated
topology on $H^*(A)$ (in the category of graded $R$-modules).
b) In analogy with a), the homology $H^*(M)$ is endowed with
the topology where a basis of neighbourhoods of zero is formed
by the kernels of the maps $H^*(M)\to H^*(M/M')$, where $M'$ runs
through a basis of neighbourhoods of zero formed by dg submodules of
finite codimension.
\end{proof}

A morphism $L \to M$ between pseudocompact dg modules is a {\em
quasi-isomorphism} if it induces an isomorphism $H^*(L) \to H^*(M)$.
Let $\cc_{pc}(A)$ be the category of pseudocompact dg right
$A$-modules. It becomes a dg category in the natural way. Let
$\ch_{pc}(A)$ be the associated zeroth homology category, \ie the
category up to homotopy of pseudocompact dg $A$-modules. The {\em
strictly perfect derived category of $A$} is the thick subcategory
of $\ch_{pc}(A)$ generated by the $A$-module $A$. A pseudocompact dg
$A$-module is {\em strictly perfect} if it lies in this subcategory.

The opposite algebra $A\op$ is still a bilaterally pseudocompact dg
algebra. If $A$ and $A'$ are two bilaterally pseudocompact dg
algebras, then so is their completed tensor product $A \hat{\ten}_k
A'$. The dg algebra $A$ is a pseudocompact dg module over the
enveloping algebra $A\op\hat{\ten} A$ (this is the reason why we use
bilaterally pseudocompact algebras).

The bilaterally pseudocompact dg algebra $A$ is {\em topologically homologically
smooth} if the module $A$ considered as a pseudocompact dg module
over $A\op\hat{\ten}_k A$ is quasi-isomorphic to a strictly perfect
dg module.

\begin{lemma}[\cite{Keller08a}] If $A$ is the completed path algebra of a finite
graded quiver endowed with a continuous differential sending each
arrow to a possibly infinite linear combination of paths of length
$\geq 2$, then $A$ is topologically homologically smooth.
\end{lemma}

We define the {\em pseudocompact derived category} $\cd_{pc}(A)$ to
be the localization of $\ch_{pc}(A)$ at the class of
quasi-isomorphisms. As we will see below, this definition does yield
a reasonable category if $A$ is topologically homologically smooth
and concentrated in non positive degrees. For arbitrary
pseudocompact dg algebras, the category $\cd_{pc}(A)$ should be
defined with more care, \confer~\cite{Positselski09}. We define the
{\em perfect derived category $\per(A)$} to be the thick subcategory
of $\cd_{pc}(A)$ generated by the free $A$-module of rank 1. We
define the {\em finite-dimensional derived category $\cd_{fd}(A)$}
to be the full subcategory whose objects are the pseudocompact dg
modules $M$ such that $(\cd_{pc}(A))(P,M)$ is finite-dimensional for
each perfect $P$.

\begin{proposition}\label{P:pseudocompact-derived-category} Assume
that $A$ is topologically homologically smooth and that $H^p(A)$
vanishes for all $p>0$.
\begin{itemize}

\item[a)] The canonical functor $\ch_{pc}(A)\to \cd_{pc}(A)$ has a
left adjoint $M \mapsto\bp M$.

\item[b)] The triangulated category $\cd_{fd}(A)$ is generated
by the dg modules of finite dimension concentrated in degree~$0$.

\item[c)] The full subcategory $\cd_{fd}(A)$ of $\cd_{pc}(A)$
is contained in the perfect derived category $\per(A)$.
In the setup of section~\ref{ss:finite-dimensional-dg-modules}, it
coincides with the subcategory of the same name defined there.

\item[d)] The opposite category $\cd_{pc}(A)\op$ is compactly
generated by $\cd_{fd}(A)$.

\item[e)] Let $A \to A'$ be a quasi-isomorphism of bilaterally pseudocompact,
topologically homologically smooth dg algebras whose homology is
concentrated in non positive degrees. Then the restriction functor
$\cd_{pc}(A') \to \cd_{pc}(A)$ is an equivalence. In particular, if
the homology of $A$ is concentrated in degree~$0$, there is an
equivalence $\cd_{pc}(A) \to \cd_{pc}(H^0 A)$. Moreover, in this
casem $\cd_{pc}(H^0 A)$ is equivalent to the derived category of the
abelian category $\Pcm(H^0 A)$.

\item[f)] Assume that $A$ is a complete path algebra as
in section~\ref{ss:finite-dimensional-dg-modules}. There is an
equivalence between $\cd_{pc}(A)\op$ and the localizing subcategory
$\cd_0(A)$ of the ordinary derived category $\cd(A)$ generated by
the finite-dimensional dg $A$-modules.

\end{itemize}
\end{proposition}

\begin{proof}
  a) Let $\phi: P \to A$ be a quasi-isomorphism of $A\op\hat{\ten}
  A$-modules where $P$ is strictly perfect. Then the underlying
  morphism of pseudocompact left $A$-modules is an homotopy
  equivalence. Indeed, its cone is contractible since it is acyclic
  and lies in the thick subcategory of $\ch_{pc}(A\op)$ generated by
  the $A\hat{\ten}_k V$, where $V$ is a pseudocompact dg $k$-module.
  Thus, for each pseudocompact dg $A$-module $M$, the induced
  morphism
\[
M\hat{\ten}_A P \to M\hat{\ten}_A A = M
\]
is a quasi-isomorphism. We claim that the canonical map
\[
\Hom_{\ch_{pc}(A)}(M\hat{\ten}_A P, ? ) \to
\Hom_{\cd_{pc}(A)}(M\hat{\ten}_A P, ? )
\]
is bijective (which shows that $M \mapsto M\hat{\ten}_A P$ is left
adjoint to the canonical functor). Indeed, since $P$ is strictly
perfect, it suffices to show that if we replace $P$ by
$A\hat{\ten}_k A$, then the corresponding map is bijective. But we
have
\[
M \hat{\ten}_A (A \hat{\ten}_k A) = M \hat{\ten}_k A
\]
and the functor
\[
\Hom_{\ch_{pc}(A)}(M\hat{\ten}_k A, ?) = \Hom_{\ch_{pc}(R)}(M, ?)
\]
makes quasi-isomorphisms invertible, which implies the claim.

b) After replacing $A$ with $\tau_{\leq 0} A$,
\confer~section~\ref{ss:t-structure} below, we may and will assume
that the component $A^p$ vanishes for all $p>0$. Then, for each dg
module $M$, in the exact sequence of complexes
\[
0 \to \tau_{< 0} M \to M \to \tau_{\geq 0} M \to 0 \ko
\]
the complex $\tau_{< 0} M$ is a dg $A$-submodule of $M$. Thus, the
sequence is a sequence of dg $A$-modules. It is clear that the
truncation functors $\tau_{<0}$ and $\tau_{\geq 0}$ preserve
quasi-isomorphisms and easy to see that they define  $t$-structures
on $\cd (A)$, $\cd_{pc}(A)$ and $\cd_{fd}(A)$,
\confer~section~\ref{ss:t-structure} below. Clearly, the
$t$-structure obtained on $\cd_{fd}(A)$ is non degenerate and each
object $M$ of its heart is quasi-isomorphic to the dg module
$H^0(M)$, which is concentrated in degree~$0$ and
finite-dimensional. This clearly implies the claim.

c) Let $M$ be in $\cd_{fd}(A)$. By b), we may assume that $M$ is of
finite total dimension. Let $P$ be as in the proof of a). Then
$M\hat{\ten}_A P$ is perfect and quasi-isomorphic to $M$.
The second claim follows because the two categories identify
with the same full subcategory of $\per(A)$.

d) We have to check that $\cd_{pc}(A)$ has arbitrary products.
Indeed, they are given by products of pseudocompact dg modules,
which in turn are given by the products of the underlying dg
modules. We have to check that $k$-finite-dimensional pseudocompact
dg modules are co-compact in $\cd_{pc}(A)$. For this, we first
observe that they are  co-compact in $\ch_{pc}(A)$. Now let $P$ be
as in the proof of a) and let $X_i$, $i\in I$, be a family of
objects of $\cd_{pc}(A)$. We consider the space of morphisms
\[
(\ch A)((\prod_{i} X_i)\hat{\ten}_A P, M).
\]
Since the module $M$ is finite-dimensional, it is annihilated by
some bilateral dg ideal $I$ of finite codimension in $A$. Thus
the above space of morphisms is isomorphic to
\[
(\ch A)((\prod_{i} X_i)\hat{\ten}_A (P/PI), M)
\]
Now the bimodule $P/PI$ is perfect as a left dg $A$-module
(since $A\op\hat{\ten}_A (A/I) = A\op\ten_A (A/I)$ is left perfect).
Thus, we have the isomorphisms
\[
(\prod_{i} X_i)\hat{\ten}_A (P/PI)= (\prod_{i} X_i)\ten_A (P/PI)
= \prod_{i} (X_i\ten_A (P/PI)).
\]
Since $M$ is co-compact in $\ch A$, we find the isomorphism
\[
(\ch A)(\prod_{i} (X_i\ten_A (P/PI)), M)
= \coprod_i (\ch A)(X_i\ten_A (P/PI), M).
\]
Now the last space is isomorphic to
\[
\coprod_i (\ch A)(X_i\hat{\ten}_A P, M) =
\coprod_i (\cd A)(X_i, M).
\]
Finally, we have to show that the left orthogonal of
$\cd_{fd}(A)$ vanishes in $\cd_{pc}(A)$. Indeed, if $M$ belongs to
the left orthogonal of $\cd_{fd}(A)$, then all the maps $M \to
M/M'$, where $M'$ is a dg submodule of finite codimension, vanish in
$\cd_{pc}(A)$. Thus the maps $H^*(M) \to H^*(M/M')$ vanish. But
$H^*(M)$ is the inverse limit of the system of the $H^*(M/M')$,
where $M'$ ranges over a basis of neighbourhoods formed by dg
submodules of finite codimension. So $H^*(M)$ vanishes and $M$
vanishes in $\cd_{pc}(A)$.

e) Let $R: \cd_{pc}(A') \to \cd_{pc}(A)$ be the restriction functor.
Let $P$ be as in the proof of a). Then $R$ admits the left adjoint
$L$ given by $M \mapsto (M\hat{\ten}_A P) \hat{\ten}_A A'$. It is
easy to check that $R$ and $L$ induce quasi-inverse equivalences
between $\per(A)$ and $\per(A')$. An object $M$ of $\per(A)$ belongs
to the subcategory $\cd_{fd}(A)$ if and only if the space
$\Hom_{\per(A)}(P,M) $ is finite-dimensional for each $P$ in
$\per(A)$. Therefore, the equivalences $L$ and $R$ must induce
quasi-inverse equivalences between $\cd_{fd}(A)$ and $\cd_{fd}(A')$.
Clearly the functor $R$ commutes with arbitrary products. Now the
claim follows from c). We obtain the equivalence between $\cd_{pc}(A)$
and $\cd_{pc}(H^0 A)$ by using the quasi-isomorphism
\[
A \leftarrow \tau_{\leq 0} A \to H^0 A.
\]
The canonical functor $\cd (\Pcm(H^0 A)) \to \cd_{pc}(H^0 A)$ is an
equivalence because its restriction to the subcategory of
finite-dimensional dg modules concentrated in degree~$0$ is an
equivalence and this subcategory compactly generates the opposites
of both categories.

f) Let $\ch_0(A)$ be the full subcategory of $\ch(A)$ formed by the dg
modules which are filtered unions of their finite-dimensional
submodules and let $\ac_0(A)$ be its full subcategory formed by the
acyclic dg modules $N$ such that each finite-dimensional dg submodule
of $N$ is contained in an acyclic finite-dimensional dg submodule.
Notice that the quotient $\ch_0(A)/\ac_0(A)$ inherits arbitary
coproducts from $\ch_0(A)$. We have a natural functor
\[
\ch_0(A)/\ac_0(A) \to \cd_0(A)
\]
and a natural duality functor
\[
\ch_0(A)/\ac_0(A) \to (\cd_{pc}(A))\op
\]
taking $M$ to $\Hom_R(M,R)$. We will show that $\ch_0(A)/\ac_0(A)$
is compactly generated by its full subcategory $\cf$ of finite-dimensional
dg modules and that the above functors induce equivalences
\[
\cf \iso \cd_{fd}(A) \mbox{ and } \cf \iso \cd_{fd}(A)\op.
\]
Since the functors
\[
\ch_0(A)/\ac_0(A) \to \cd_0(A) \mbox{ and }
\ch_0(A)/\ac_0(A) \to (\cd_{pc}(A))\op
\]
both commute with arbitrary coproducts, they must be equivalences
because $\cd_{fd}(A)$ compactly generates $\cd_0(A)$
(by Theorem~\ref{thm:findim-compact}) and its opposite compactly generates
$\cd_{pc}(A)\op$ (by part d). So let us show that $\cf$ consists
of compact objects. Indeed, it follows from the definition of
$\ac_0(A)$ that for an object $F$ of $\cf$, the category of
$\ac_0(A)$-quasi-isomorphisms $F' \to F$ with finite-dimensional
$F'$ is cofinal in the category of all $\ac_0(A)$-quasi-isomorphisms
with target $F$. Now it follows from the calculus of fractions that
$F$ is compact in $\ch_0(A)/\ac_0(A)$. Moreover, we see that
$\cf$ is naturally equivalent to the category $\ch_{fd}(A)/\ac_{fd}(A)$
of section~\ref{ss:finite-dimensional-dg-modules}. Now the fact that both
functors
\[
\cf \iso \cd_{fd}(A) \mbox{ and } \cf \iso \cd_{fd}(A)\op
\]
are equivalences follows from part a) of Theorem~\ref{thm:findim-compact}.
\end{proof}

\subsection{The Calabi-Yau property} \label{ss:CY-property}
Keep the assumptions on
$k$, $R$, $A$ from section~\ref{ss:extension-to-dg-setting}. In
particular, $A$ is assumed to be topologially homologically smooth.
For two objects $L$ and $M$ of $\cd_{pc}(A)$, define
\[
\RHom_A(L,M) = \Hom_A(\bp L, M).
\]
Put $A^e = A\op\hat{\ten}_k A$ and define
\[
\Omega= \RHom_{A^e}(A,A^e).
\]
\begin{lemma}[\cite{Keller08d}] We have a canonical isomorphism
\[
D\Hom(L,M) = \Hom(M\hat{\ten}_A \Omega, L)
\]
bifunctorial in $L\in\cd_{fd}(A)$ and $M\in\cd_{pc}(A)$.
\end{lemma}

Let $n$ be an integer. For an object $L$ of $\cd_{pc}(A^e)$, define
\[
L^\# = \RHom_{A^e}(L,A^e)[n].
\]
The dg algebra $A$ is {\em $n$-Calabi-Yau as a bimodule} if there is
an isomorphism
\[
f: A \iso A^\#
\]
in $\cd_{cd}(A^e)$ such that $f^\# = f$. The preceding lemma implies that
in this case, the category $\cd_{fd}(A)$ is $n$-Calabi-Yau as a
triangulated category.

\begin{theorem}[\cite{Keller08a}]\label{T:ginzburg-smooth} Completed Ginzburg algebras are topologically
homologically smooth and $3$-Calabi-Yau.
\end{theorem}

\subsection{$t$-structure} \label{ss:t-structure} Keep the assumptions on
$k$, $R$, $A$ from the preceding section. Let us assume that
$H^n(A)$ vanishes for all $n>0$. We have the truncated complex
$\tau_{\leq 0}A$, which becomes a dg algebra in a natural way. The
dg algebra $\tau_{\leq 0} A$ is still a bilaterally pseudo-compact
and topologically homologically smooth. Indeed, we have $\tau_{\leq
0}(A/I) = \tau_{\leq 0}(A)/\tau_{\leq 0} I$ for each dg ideal $I$ of
$A$, which yields that $\tau_{\leq 0} A$ is bilaterally
pseudo-compact. Moreover, the morphism $\tau_{\leq 0} A \to A$ is a
quasi-isomorphism, which yields the homological smoothness. Let us
therefore assume that the components $A^n$ vanish in all degrees
$n>0$.

\begin{lemma} \label{lemma:t-structure}
\begin{itemize}
\item[a)] For each pseudocompact dg $A$-module $M$, the
truncations $\tau_{\leq 0} M$  and $\tau_{>0}M$ are pseudocompact dg
$A$-modules.
\item[b)] The functors $\tau_{\leq 0}$ and $\tau_{>0}$
define a $t$-structure on $\cd_{pc}(A)$.
\item[c)] The heart of the $t$-structure is canonically
equivalent to the category $\Pcm(H^0 (A))$ of pseudocompact modules
over the pseudocompact algebra $H^0(A)$.
\end{itemize}
\end{lemma}

\begin{proof} a) Let $M'$ be a dg submodule of $M$. The canonical
map
\[
Z^0(M)/Z^0(M') \to Z^0(M/M')
\]
is an isomorphism because the map $B^1(M') \to B^1(M)$ is injective.
Thus, we have an isomorphism
\[
\tau_{\leq 0}M = \lim (\tau_{\leq 0}(M) / \tau_{\leq 0}(M')) \ko
\]
where $M'$ ranges through a basis of neighbourhoods of $0$ formed by
dg submodules of finite codimension. It follows that $\tau_{\leq
0}M$ is endowed with a complete separated topology and is
pseudocompact. Similarly for $\tau_{>0}M = M/\tau_{\leq 0} M$.

b) They define functors in the homotopy categories which preserve
quasi-isomorphisms and thus descend to $\cd_{pc}(A)$. As in the
proof of lemma~\ref{lemma:canonical-t-structure}, it is not hard to
check that the induced functors define a $t$-structure.

c) We omit the proof, which is completely analogous to that of part
b) of lemma~\ref{lemma:canonical-t-structure}.
\end{proof}

\subsection{Jacobi-finite quivers with potentials}
\label{ss:Jacobi-finite-quivers-with-potentials} Let $k$ be a field
and $R$ a finite-dimensional separable $k$-algebra. Let $A$ be a dg
$R$-algebra which is bilaterally pseudocompact
(section~\ref{ss:extension-to-dg-setting}). Assume moreover that
\begin{itemize}
\item[1)] $A$ is topologically homologically smooth
(section~\ref{ss:extension-to-dg-setting}),
\item[2)] for each $p>0$, the space $H^p A$ vanishes,
\item[3)] the algebra $H^0(A)$ is finite-dimensional,
\item[4)] $A$ is $3$-Calabi-Yau as a bimodule
(section~\ref{ss:CY-property}).
\end{itemize}
Thanks to assumption 1), the finite-dimensional derived
category $\cd_{fd}(A)$ is contained in the perfect derived
category. The {\em generalized cluster category} \cite{Amiot08a} is
the triangle quotient $\cc=\per(A)/\cd_{fd}(A)$. Let
$\pi$ be the canonical projection functor.

\begin{theorem} The generalized cluster category $\cc$ is
$\Hom$-finite and $2$-Calabi-Yau. The object $\pi(A)$ is
cluster-tilting in $\cc$.
\end{theorem}

This theorem was proved for non topological dg algebras $A$ by
C.~Amiot, \confer Theorem 2.1 of \cite{Amiot08a}. Using the material
of the preceding sections, one can easily imitate Amiot's proof to
obtain the above topological version. We leave the details to the
reader. In particular, by Theorem~\ref{T:ginzburg-smooth}, the
statement holds for complete Ginzburg dg algebras associated with
quivers with potential $(Q,W)$, where $W$ belongs to the completed
path algebra of $Q$ and the Jacobian algebra is finite-dimensional.

\subsection{Mutation} \label{ss:pseudocompact-main-theorem}
Let $(Q,W)$ be a quiver with potential and
let $\Gamma$ be the associated completed Ginzburg dg algebra. By
definition, $\Gamma$ is concentrated in non positive degrees and
by Theorem~\ref{T:ginzburg-smooth}, it is topologically homologically smooth
and $3$-Calabi-Yau. Let $i$ be a vertex of $Q$ such that conditions
(c1), (c2) and (c3) of section~\ref{S:mutation} hold. Let $(Q',W')$
be the mutation $\tilde{\mu}_i(Q,W)$ and $\Gamma'$ the associated
Ginzburg dg algebra.

\begin{theorem}\label{T:mainthm-complete}
\begin{itemize}
\item[a)] There is a triangle equivalence
\[
F_{pc}:\mathcal{D}_{pc}(\Gamma')\longrightarrow\mathcal{D}_{pc}(\Gamma)
\]
which sends the $P'_j$ to $P_j$ for $j\neq i$ and to the cone $T_i$
over the morphism
\[
P_i\longrightarrow \bigoplus_{\alpha}P_{t(\alpha)}
\]
for $i=j$, where we have a summand $P_{t(\alpha)}$ for each arrow
$\alpha$ of $Q$ with source $i$ and the corresponding component of
the map is the left multiplication by $\alpha$. The functor $F_{pc}$
restricts to triangle equivalences from $\per(\Gamma')$ to
$\per(\Gamma)$ and from $\mathcal{D}_{fd}(\Gamma')$ to
$\mathcal{D}_{fd}(\Gamma)$.

\item[b)] Let $\Gamma_{\mathrm{red}}$ respectively $\Gamma'_{\mathrm{red}}$ be
the complete Ginzburg dg algebra associated with the reduction of
$(Q,W)$ respectively the reduction $\mu_i(Q,W)$ of
$\tilde{\mu}_i(Q,W)$. The functor $F_{pc}$ yields a triangle equivalence
\[
F_{\mathrm{red}}:\mathcal{D}_{pc}(\Gamma'_{\mathrm{red}})\longrightarrow\mathcal{D}_{pc}(\Gamma_{\mathrm{red}}),
\]
which restricts to triangle equivalences from
$\per(\Gamma'_{\mathrm{red}})$ to $\per(\Gamma_{\mathrm{red}})$ and
from $\mathcal{D}_{fd}(\Gamma'_{\mathrm{red}})$ to
$\mathcal{D}_{fd}(\Gamma_{\mathrm{red}})$.
\end{itemize}
\end{theorem}

\begin{remark} By combining the theorem with the facts on $t$-structures from
section~\ref{ss:t-structure} we obtain a comparison of the categories
of pseudocompact modules
\[
\Pcm(H^0(\Gamma)) \mbox{ and } \Pcm(H^0(\Gamma'))
\]
completely analogous to the one in section~\ref{ss:comparison-of-t-structures}.
\end{remark}

\begin{proof} a) Let $X$ be the $\Gamma'$-$\Gamma$-bimodule constructed
in section~\ref{ss:bimodule}. It gives rise to a pair of adjoint
functors: the left adjoint $F_{pc}: \cd_{pc}(\Gamma') \to
\cd_{pc}(\Gamma)$ is the left derived functor of the completed
tensor product functor $M \mapsto M\hat{\ten}_{\Gamma'} X$. The
right adjoint $G_{pc}$ is induced by the functor
$\Hom_{\Gamma}(X,?)$. These functors restrict to a pair of adjoint
functors between $\cd_{fd}(\Gamma)$ and $\cd_{fd}(\Gamma')$. The
induced functors are isomorphic to those of
Proposition~\ref{P:equivfd}. Indeed, by
Proposition~\ref{P:pseudocompact-derived-category}, the categories
$\cd_{fd}$ are contained in the perfect derived categories and the
canonical morphism from the left derived functor of the (non
completed) tensor product $M \mapsto M\ten_{\Gamma'} X$ to $F_{pc}$
restricts to an isomorphism on the perfect derived category and
similarly for $G_{pc}$. Thus, the restriction of $G_{pc}$ is an
equivalence from $\cd_{fd}(\Gamma)$ to $\cd_{fd}(\Gamma')$. Moreover,
the functor $G_{pc}$ commutes with arbitrary products since it
admits a left adjoint. It follows that $G_{pc}$ is an equivalence
since the opposites of the categories $\cd_{pc}$ are compactly
generated by the subcategories $\cd_{fd}$, by
Proposition~\ref{P:pseudocompact-derived-category}.

b) This follows from a) and part e) of
Proposition~\ref{P:pseudocompact-derived-category}.
\end{proof}


\begin{thebibliography}{10}

\bibitem{Amiot08a}
Claire Amiot, \emph{Cluster categories for algebras of global dimension $2$ and
  quivers with potential}, arXiv:0805.1035, to appear in Annales Scient.
  Fourier.

\bibitem{AuslanderPlatzeckReiten79}
Maurice Auslander, Mar{\'\i}a~In{\'e}s Platzeck, and Idun Reiten, \emph{Coxeter
  functors without diagrams}, Trans. Amer. Math. Soc. \textbf{250} (1979),
  1--46.

\bibitem{BuanIyamaReitenSmith08}
Aslak Bakke~Buan, Osamu Iyama, Idun Reiten, and David Smith, \emph{Mutation of
  cluster-tilting objects and potentials}, arXiv:0804.3813.

\bibitem{BuanMarshReinekeReitenTodorov06}
Aslak Bakke~Buan, Robert~J. Marsh, Markus Reineke, Idun Reiten, and Gordana
  Todorov, \emph{Tilting theory and cluster combinatorics}, Advances in
  Mathematics \textbf{204 (2)} (2006), 572--618.

\bibitem{BuanMarshReiten04}
Aslak Bakke~Buan, Robert~J. Marsh, and Idun Reiten, \emph{Cluster-tilted
  algebras}, Trans. Amer. Math. Soc., \textbf{359} (2007), no.~1, 323--332,
  electronic.

\bibitem{BeasleyPlesser01}
Chris~E. Beasley and M.~Ronen Plesser, \emph{Toric duality is {S}eiberg
  duality}, J. High Energy Phys. (2001), no.~12, Paper 1, 38, hep-th/0109053.

\bibitem{BerensteinDouglas02}
David Berenstein and Michael Douglas, \emph{Seiberg duality for quiver gauge
  theories}, arXiv:hep-th/0207027v1.

\bibitem{BernsteinGelfandPonomarev73}
I.~N. Bern{\v{s}}te\u{\i}n, I.~M. Gel{\cprime}fand, and V.~A. Ponomarev,
  \emph{Coxeter functors, and {G}abriel's theorem}, Uspehi Mat. Nauk
  \textbf{28} (1973), no.~2(170), 19--33.

\bibitem{BondalKapranov89}
A.~I. Bondal and M.~M. Kapranov, \emph{Representable functors, {S}erre
  functors, and reconstructions}, Izv. Akad. Nauk SSSR Ser. Mat. \textbf{53}
  (1989), no.~6, 1183--1205, 1337.

\bibitem{Bridgeland05}
Tom Bridgeland, \emph{t-structures on some local {C}alabi-{Y}au varieties}, J.
  Algebra \textbf{289} (2005), no.~2, 453--483.

\bibitem{DerksenWeymanZelevinsky09}
Harm Derksen, Jerzy Weyman, and Andrei Zelevinsky, \emph{Quivers with
  potentials and their representations {II}: {Applications to cluster
  algebras}}, arXiv:0904.0676v1.

\bibitem{DerksenWeymanZelevinsky08}
\bysame, \emph{Quivers with potentials and their representations {I}:
  {Mutations}}, Selecta Mathematica \textbf{14} (2008), 59--119.

\bibitem{Ebeling81}
Wolfgang Ebeling, \emph{Quadratische {F}ormen und {M}onodromiegruppen von
  {S}ingularit\"aten}, Math. Ann. \textbf{255} (1981), no.~4, 463--498.

\bibitem{FengHananyHeIqbal03}
Bo~Feng, Amihay Hanany, Yang~Hui He, and Amer Iqbal, \emph{Quiver theories,
  soliton spectra and {P}icard-{L}efschetz transformations}, J. High Energy
  Phys. (2003), no.~2, 056, 33, hep-th/0206152.

\bibitem{FengHananyHeUranga01}
Bo~Feng, Amihay Hanany, Yang-Hui He, and Angel~M. Uranga, \emph{Toric duality
  as {S}eiberg duality and brane diamonds}, J. High Energy Phys. (2001),
  no.~12, Paper 35, 29, hep-th/0109063.

\bibitem{FominZelevinsky02}
Sergey Fomin and Andrei Zelevinsky, \emph{Cluster algebras. {I}.
  {F}oundations}, J. Amer. Math. Soc. \textbf{15} (2002), no.~2, 497--529
  (electronic).

\bibitem{Franke97}
Jens Franke, \emph{On the {Brown} representability theorem for triangulated
  categories}, Topology \textbf{40} (2001), no.~4, 667--680.

\bibitem{Gabriel62}
Peter Gabriel, \emph{Des cat{\'e}gories ab{\'e}liennes}, Bull. Soc. Math.
  France \textbf{90} (1962), 323--448.

\bibitem{Gabrielov73}
A.~M. Gabri{\`e}lov, \emph{Intersection matrices for certain singularities},
  Funkcional. Anal. i Prilo\v zen. \textbf{7} (1973), no.~3, 18--32.

\bibitem{GeissLeclercSchroeer06}
Christof Gei\ss, Bernard Leclerc, and Jan Schr{\"o}er, \emph{Rigid modules over
  preprojective algebras}, Invent. Math. \textbf{165} (2006), no.~3, 589--632.

\bibitem{Ginzburg06}
Victor Ginzburg, \emph{{Calabi-Yau} algebras}, arXiv:math/0612139v3 [math.AG].

\bibitem{Happel87}
Dieter Happel, \emph{On the derived category of a finite-dimensional algebra},
  Comment. Math. Helv. \textbf{62} (1987), no.~3, 339--389.

\bibitem{HappelReitenSmaloe96}
Dieter Happel, Idun Reiten, and Sverre~O. Smal{\o}, \emph{Tilting in abelian
  categories and quasitilted algebras}, Mem. Amer. Math. Soc. \textbf{120}
  (1996), no.~575, viii+ 88.

\bibitem{HappelUnger05}
Dieter Happel and Luise Unger, \emph{On a partial order of tilting modules},
  Algebr. Represent. Theory \textbf{8} (2005), no.~2, 147--156.

\bibitem{IyamaReiten06}
Osamu Iyama and Idun Reiten, \emph{{Fomin-Zelevinsky mutation and tilting
  modules over Calabi-Yau algebras}}, arXiv:math.RT/0605136, to appear in Amer.
  J. Math.

\bibitem{KashiwaraSchapira05}
Masaki Kashiwara and Pierre Schapira, \emph{Categories and sheaves},
  Grundlehren der Mathematischen Wissenschaften [Fundamental Principles of
  Mathematical Sciences], vol. 332, Springer-Verlag, Berlin, 2005.

\bibitem{Keller08c}
Bernhard Keller, \emph{Cluster algebras, quiver representations and
  triangulated categories}, arXive:0807.1960.

\bibitem{Keller08a}
\bysame, \emph{Deformed {C}alabi-{Y}au completions}, in preparation.

\bibitem{Keller94}
\bysame, \emph{Deriving {D}{G} categories}, Ann. Sci. {\'E}cole Norm. Sup. (4)
  \textbf{27} (1994), no.~1, 63--102.

\bibitem{Keller06d}
\bysame, \emph{On differential graded categories}, International Congress of
  Mathematicians. Vol. II, Eur. Math. Soc., Z\"urich, 2006, pp.~151--190.

\bibitem{Keller08d}
\bysame, \emph{Triangulated {C}alabi-{Y}au categories}, Trends in
  Representation Theory of Algebras (Zurich) (A.~Skowro\'nski, ed.), European
  Mathematical Society, 2008, pp.~467--489.

\bibitem{KellerReiten07}
Bernhard Keller and Idun Reiten, \emph{{Cluster-tilted algebras are Gorenstein
  and stably Calabi-Yau}}, Advances in Mathematics \textbf{211} (2007),
  123--151.

\bibitem{KellerVossieck88}
Bernhard Keller and Dieter Vossieck, \emph{Aisles in derived categories}, Bull.
  Soc. Math. Belg. S{\'e}r. A \textbf{40} (1988), no.~2, 239--253.

\bibitem{KlebanovStrassler00}
Igor~R. Klebanov and Matthew~J. Strassler, \emph{Supergravity and a confining
  gauge theory: duality cascades and {$\chi\,SB$}-resolution of naked
  singularities}, J. High Energy Phys. (2000), no.~8, Paper 52, 35,
  hep-th/0007191.

\bibitem{KontsevichSoibelman08}
Maxim Kontsevich and Yan Soibelman, \emph{Stability structures,
  {D}onaldson-{T}homas invariants and cluster transformations},
  arXiv:0811.2435.

\bibitem{MarshReinekeZelevinsky03}
Robert Marsh, Markus Reineke, and Andrei Zelevinsky, \emph{Generalized
  associahedra via quiver representations}, Trans. Amer. Math. Soc.
  \textbf{355} (2003), no.~10, 4171--4186 (electronic).

\bibitem{MukhopadhyayRay04}
Subir Mukhopadhyay and Koushik Ray, \emph{Seiberg duality as derived
  equivalence for some quiver gauge theories}, J. High Energy Phys. (2004),
  no.~2, 070, 22 pp. (electronic).

\bibitem{Neeman92a}
Amnon Neeman, \emph{The connection between the {K--theory} localisation theorem
  of {Thomason}, {Trobaugh} and {Yao}, and the smashing subcategories of
  {Bousfield} and {Ravenel}}, Ann. Sci. {\'E}cole Normale Sup{\'e}rieure
  \textbf{25} (1992), 547--566.

\bibitem{Neeman99}
\bysame, \emph{Triangulated {Categories}}, Annals of Mathematics Studies, vol.
  148, Princeton University Press, Princeton, NJ, 2001.

\bibitem{Plamondon09}
Pierre-Guy Plamondon, Ph.~D.~thesis in preparation, 2009.

\bibitem{Positselski09}
Leonid Positselski, \emph{Two kinds of derived categories, {K}oszul duality,
  and comodule-contramodule correspondence}, preprint available at the author's
  homepage http://positselski.narod.ru.

\bibitem{RiedtmannSchofield91}
Christine Riedtmann and Aidan Schofield, \emph{On a simplicial complex
  associated with tilting modules}, Comment. Math. Helv. \textbf{66} (1991),
  no.~1, 70--78.

\bibitem{Ringel07}
Claus~Michael Ringel, \emph{Some remarks concerning tilting modules and tilted
  algebras. {Origin. Relevance. Future.}}, Handbook of Tilting Theory, LMS
  Lecture Note Series, vol. 332, Cambridge Univ. Press, Cambridge, 2007,
  pp.~49--104.

\bibitem{Seiberg95}
N.~Seiberg, \emph{Electric-magnetic duality in supersymmetric non-abelian gauge
  theories}, Nuclear Phys. B \textbf{435} (1995), no.~1-2, 129--146,
  arXiv:hep-th/9411149.

\bibitem{Simson07a}
Daniel Simson, \emph{Letter to {B.~Keller}}, October 26, 2007.

\bibitem{Simson07}
\bysame, \emph{Path coalgebras of profinite bound quivers, cotensor coalgebras
  of bound species and locally nilpotent representations}, Colloq. Math.
  \textbf{109} (2007), no.~2, 307--343.

\bibitem{VandenBergh08}
Michel Van~den Bergh, \emph{The signs of {S}erre duality}, Appendix A to
  R.~Bocklandt, Graded Calabi-Yau algebras of dimension 3, Journal of Pure and
  Applied Algebra 212 (2008), 14--32.

\bibitem{Vitoria09}
Jorge Vit{\'o}ria, \emph{Mutations vs. {S}eiberg duality}, J. Algebra
  \textbf{321} (2009), no.~3, 816--828.

\end{thebibliography}

\def\cprime{$'$} \def\cprime{$'$}
\providecommand{\bysame}{\leavevmode\hbox to3em{\hrulefill}\thinspace}
\providecommand{\MR}{\relax\ifhmode\unskip\space\fi MR }
\providecommand{\MRhref}[2]{%
  \href{http://www.ams.org/mathscinet-getitem?mr=#1}{#2}
}
\providecommand{\href}[2]{#2}

\end{document}